\documentclass[10pt,reqno]{amsart}

\usepackage{graphicx}

\usepackage{amssymb}
\usepackage{amsthm}
\theoremstyle{plain}
\usepackage{comment}
\usepackage{enumitem} 
\usepackage{amsmath,amssymb,amsthm,mathtools}
\usepackage{dsfont} \usepackage{mathabx}
\usepackage{bm}
\usepackage[hidelinks]{hyperref}
\usepackage[backend=biber,sorting=ydnt]{biblatex}
\newtheorem*{theorem*}{Theorem}

\newtheorem{definition}{Definition}
\newtheorem{theorem}{Theorem}
\newtheorem{proposition}{Proposition}
\newtheorem{corollary}{Corollary}
\newtheorem{conjecture}{Conjecture}
\newtheorem{lemma}{Lemma}

\theoremstyle{remark}
\newtheorem{remark}{Remark}
\usepackage{lineno}

\title{Refinements on the Restriction and Kakeya problems range in $\mathbb{R}^4$}

\author[T. Moreira]{Tiago Moreira}
\address{Scuola Internazionale Superiore di Studi Avanzati (SISSA), Via Bonomea 265, 34136 Trieste, Italy}
\email{tdiasmor@sissa.it}

\begin{document}

\begin{abstract}
We obtain improved Fourier restriction and Kakeya maximal estimates in $\mathbb R^4$. More precisely, by proving the two-ends incidence estimates $\mathbf{TE}\left(\frac{269}{89},
\frac{179}{89},
\frac{493}{534}\right)$ and $
\mathbf{TE}(3.06\ldots,3.06\ldots,0.94\ldots)$, we extend the restriction range to $p>2+ \frac{356}{447}$ and the Kakeya maximal estimate up to dimension $3.06056$. The implications to the known lower-bounds for the Hausdorff dimension of Kakeya sets in $\mathbb{R}^4$ and the Bochner--Riesz conjecture are also studied.
\end{abstract}

\maketitle
\section{Introduction}
A recurring principle in Fourier analysis is that analytic estimates for oscillatory
functions are governed by geometric information about where their oscillations concentrate. The Fourier restriction problem provides one of the most prominent
examples of this phenomenon. After localization in frequency, an oscillatory integral associated with a curved hypersurface can be decomposed into wave-packets, each of which is essentially supported on a long thin tube in physical space. Consequently, estimates for the original operator are closely tied to the possible intersection patterns of large collections of such tubes. This viewpoint places
the restriction problem in direct contact with Kakeya-type incidence geometry.

To make this precise, let
    $S=\{(\xi,h(\xi)):\xi\in B^{n-1}(0,1)\}\subset\mathbb{R}^n$
be a hypersurface with positive definite second fundamental form. The Fourier
extension operator associated with $S$ is defined as 
\begin{equation}\label{eq:extension}
    E_Sf(x,t)
    =
    \int_{B^{n-1}(0,1)}
    e^{2\pi i(x\cdot\xi+th(\xi))}
    f(\xi)\,d\xi .
\end{equation}
One of the central questions is to determine the range of exponents for which
$E_Sf$ satisfies uniform $L^p$-estimates. The expected answer is given by Stein's restriction conjecture first proposed in \cite{Ste79}.

\begin{conjecture}[Restriction conjecture]\label{conj:restriction}
Let $S\subset\mathbb{R}^n$ be a compact $C^2$ hypersurface with positive
definite second fundamental form.  If
    $p>\frac{2n}{n-1}$,
then
\begin{equation}\label{eq:restriction}
    \|E_Sf\|_{L^p(\mathbb{R}^n)}
    \lesssim_p
    \|f\|_{L^p(B^{n-1}(0,1))} .
\end{equation}
\end{conjecture}

The conjecture is known in dimension two, following the work of Fefferman \cite{Fef70}, but remains open in dimensions $n\geq3$. Much of the recent
progress has come from methods that exploit the geometry of wave-packets rather than treating the extension operator solely through oscillatory-integral estimates. Guth's introduction of polynomial partitioning into the restriction
problem in \cite{Gut16} was particularly influential. In dimension three, it led to the range $p>3+\frac14$. 
\newpage
Wang later introduced the broom method and improved this to  $p>3+\frac{3}{13}$ in \cite{Wan22}. More recently, Wang and Wu developed a different mechanism that combines refined decoupling with incidence estimates for families of wave-packet tubes in \cite{WW24}.  Their approach yields the restriction estimate in $\mathbb{R}^3$ for $p>\frac{22}{7}$.

The purpose of the present work is to further investigate this interaction between decoupling and incidence geometry in four dimensions. The restriction problem in $\mathbb{R}^4$ occupies an interesting intermediate regime: polynomial partitioning gives strong structural information, but the geometry of tube intersections becomes considerably richer than in three dimensions. Previous
results illustrate how difficult it is to obtain even small improvements in the exponent. Guth proved the restriction estimate for $p>2+\frac45$ in \cite{Gut18}. This was subsequently improved by Demeter to $p>2+\frac{66642}{83303}$ in \cite{Dem19}, and by Hickman and Rogers to $p>2+\frac{1407}{1759}$ in \cite{HR19}. For higher-dimensional related developments, we guide the reader to \cite{HZ24,GWZ24}.

Our approach relies similarly on the work done by Borges et al. in \cite{BCCLXZ}, which combined two ideas that were originally developed in somewhat
different contexts.  The first is the planebrush argument of Katz and Zahl \cite{KZ}, which exploits the additional structure present when many tubes cluster near two-dimensional planes. The second is the decoupling-incidence
framework of Wang and Wu \cite{WW24}, in which geometric information about tube overlaps is converted into better analytic estimates through refined decoupling. Bringing these mechanisms together, Borges et al. obtained an improved range $p>2+\frac{200}{251}$. We note that besides the work being rather striking, it does not use the latent parameter $M$ present in the main construction (originally developed by Katz--Zahl). We recover this parameter and, by carefully developing a new estimate using it, an interpolation between the three (the before mentioned, planebrush and trilinear) estimates gives a strictly greater range.
\begin{theorem}\label{thm:main-restriction}
For every $p>2+\frac{356}{447}$,
one has
\begin{equation}
\label{restriction}
    \|E_Sf\|_{L^p(\mathbb{R}^4)}
    \lesssim_p
    \|f\|_{L^p(B^3(0,1))}.
\end{equation}
\end{theorem}
Numerically,
    $2+\frac{200}{251}\approx 2.79681$, whereas $2+\frac{356}{447}\approx 2.79642$
so Theorem~\ref{thm:main-restriction} improves the previously known
four-dimensional range. 

An advantage of the decoupling-incidence viewpoint is that it is not restricted to the translation-invariant extension operator. The same general framework can be applied to H\"ormander-type oscillatory integral operators by
\cite{GWX25}. In particular, the incidence improvement obtained here has consequences for the Bochner--Riesz problem.  Recall that the Bochner--Riesz operator and its multiplier are given by
\[
    B_{\delta}f(x)
    =
    \int_{\mathbb{R}^n}
    e^{ix\cdot\xi}
    (1-|\xi|^2)_+^\delta
    \widehat f(\xi)\,d\xi, \qquad m_\delta(\xi)=(1-|\xi|^2)_+^\delta.
\]
The conjecture can be stated as follows:
\begin{conjecture}[Bochner--Riesz conjecture] 
Let $p'$ denote the Hölder conjugate of $p$. Then, for every $1<p<\infty$, $B_\delta$ is bounded on $L^p(\mathbb R^4)$ whenever the parameter $\delta$ satisfies
$\delta>
\max\left\{4\left|\frac12-\frac1p\right|-\frac12,0\right\}$.
\end{conjecture}
In dimension four, Guo et al. in \cite{GOW25} proved the conjectured Bochner--Riesz range under the condition $\max\{p,p'\}>\frac{14}{5}$ by revisiting older approaches to tackle this problem. As a consequence of the restriction argument in \cite{BCCLXZ}, this range was then enlarged to $\max\{p,p'\}\geq 2+\frac{200}{251}$. By similar procedures, our result also translates into a new improvement on this problem, reaching the range $\max\{p,p'\}\geq 2+\frac{356}{447}$.

Due to the fact that our interpolation method is able to produce incidence estimates $\mathbf{TE}\left(d,a,b\right)$ with a better factor of $b$ satisfying $b\leq\frac{d-1}{3}$, in contrast to \cite{BCCLXZ}, we are also able to study the $L^q\to L^p$ estimates away from the diagonal, as done in \cite{GWX25}. We reproduce some of their ideas in a $b$-dependent way and upgrade the best known bound for $n=4$ until now $2.8$ presented in Corollary~$1.2$ of \cite{GWX25}.
\begin{corollary}
\label{offdiagresult}
Let $p\ge \frac{786}{281}$, and let $q$ satisfy $q'\le \frac35p$.
Then, for every $\varepsilon>0$ and every $R\ge1$,
\begin{equation}
\|E_S f\|_{L^p(B(0,R))}
\lesssim_\varepsilon
R^\varepsilon
\|f\|_{L^q(B^3(0,1))}.
\end{equation}
\end{corollary}

The geometric content of the argument becomes clearer when viewed through the Kakeya problem. Since a wave-packet is concentrated on a thin tube, a large value of $E_Sf$ may arise when many such tubes overlap.  Restriction estimates therefore require quantitative control on how efficiently tubes pointing in
different directions can concentrate in a small region. Kakeya theory isolates precisely this geometry.

A Kakeya set in $\mathbb{R}^n$ is a set containing a unit line segment in every direction. The Kakeya conjecture predicts that such a set must have full Hausdorff dimension.

\begin{conjecture}[Kakeya conjecture]\label{conj:kakeya}
If $K\subset\mathbb{R}^n$ is a Kakeya set, then
    $\dim_{\mathbb H}K=n$.
\end{conjecture}

The conjecture was proved in two dimensions by Davies in \cite{Dav71}. Recently, Wang and Zahl resolved the three-dimensional case in \cite{WZ25}.  
The four-dimensional Kakeya problem has a long history. Drury obtained the Hausdorff dimension lower-bound $\frac{5}{2}$ in \cite{Dru83}. Wolff's hairbrush argument then raised the corresponding bound to $3$ in \cite{Wol95} by means of approaching a stronger quantitative version of this conjecture, denoted by Kakeya maximal conjecture. The statement goes as follows:
\begin{conjecture}[Kakeya maximal estimate at dimension $d$]
Let $\mathbb{T}$ be a set of $\delta$-tubes in $\mathbb{R}^n$ that point in
$\delta$-separated directions. Then, for each $d \leq n$ and $\varepsilon > 0$,
there exists a constant $C_\varepsilon > 0$ such that
\begin{equation}
\left\|
\sum_{T \in \mathbb{T}} \mathds{1}_T
\right\|_{L^\frac{d}{d-1}(\mathbb{R}^n)}
\leq
C_\varepsilon
\left(\frac{1}{\delta}\right)^{\frac{n}{d}-1+\varepsilon}.
\end{equation}
\end{conjecture}

Breaking this threshold requires
understanding configurations in which the tubes possess an additional algebraic or planar structure. Guth and Zahl introduced polynomial Wolff axioms and obtained a Kakeya maximal estimate at dimension $3+\frac1{40}$ in
\cite{Zah18,GZ18}. Katz and Zahl subsequently developed the planebrush argument, leading to a Kakeya maximal exponent of roughly $3.049$ and a Hausdorff dimension lower-bound of approximately $3.059$ in \cite{KZ}. Combining the planebrush decomposition with the two-ends incidence mechanism, Borges et al. also in \cite{BCCLXZ} obtained a strictly better exponent $d=\frac{159+\sqrt{145}}{56}\approx 3.05431$ for the Kakeya maximal estimate. By combining similar interpolation methods as in the proof of Theorem~\ref{thm:main-restriction} and a general bootstrap method based on the one proved by Katz--Zahl, we obtained a better dimensional exponent.
\begin{theorem}
\label{vera3}
Let $\mathbb{T}$ be a set of $\delta$-tubes in $\mathbb{R}^4$ that point in
$\delta$-separated directions. Then, for each $\varepsilon > 0$, there
exists a constant $C_\varepsilon$ such that
\begin{equation}
\left\|
\sum_{T \in \mathbb{T}} \mathds{1}_T
\right\|_{L^\frac{d}{d-1}(\mathbb{R}^4)}
\leq
C_\varepsilon
\left(\frac{1}{\delta}\right)^{\frac{4}{d}-1+\varepsilon},
\end{equation}
where
$d =
\frac{
3821166286084225+
86718963240\sqrt{17665}
}{
1252282159099828
}$.
\end{theorem}
\newpage
As a consequence of the previous result and Katz--Zahl estimates, we also get
\begin{corollary}
\label{Kakeyalowerbound}
For every Kakeya set in $\mathbb{R}^4$ one has that its Hausdorff dimension is at least 
$\frac{46480643+10717\sqrt{17665}}{15629688}
\geq
3.065.$
\end{corollary}
Moreover, if we focus our attention only on the sticky case, the conjecture was recently solved by Wang and Zakharov in \cite{wang2026kakeya}, improving the lower-bound $3.25$ previously established by Choudhuri in \cite{choudhuri2024improved}.

In the end of each main result, that is, Theorem~\ref{thm:main-restriction}, Theorem~\ref{vera3} and part of Corollary~\ref{offdiagresult}, we also prove that with the upper-bounds on the average multiplicity developed in the literature to the present day, each of them is ``essentially" sharp.
\\
\noindent\textbf{Notation.}
\begin{itemize}
    \item We write $A\lesssim B$ if $A\le CB$ for a constant $C$ depending only on fixed parameters. Lower-scripts are used when we want to make explicit the parameters that this constant depends on. All implicit constants in this case are independent of $\delta$;
    \item With $\delta\leq 1$, $A\lessapprox B$ means that, for every $\varepsilon>0$, there exists a constant $C_{\varepsilon}$,
depending only on $\varepsilon$, such that $A \le C_{\varepsilon} \delta^{-\varepsilon} B$;
\item Moreover, $A\gtrsim B$ and $A\gtrapprox B$ are defined similarly for the reverse inequality, that is $B\lesssim A$ and $B\lessapprox A$, respectively. If $A\lesssim B$ and $A\gtrsim B$, then $A\sim B.$ Identically, if one has $A\lessapprox B$ and $A \gtrapprox B$, we say $A\approx B$. We also use the latter symbol to denote numerical approximations; 
\item For $\Omega\subset \mathbb{R}^4$, we denote by $|\Omega|$ the $(\dim_{\mathbb{H}}\Omega)$-Hausdorff measure of $\Omega.$ For a finite set $E$, we define the quantity $\# E$ to be its cardinality;
\item We denote $B^k(a,\varrho)$ to be the $k$-dimensional ball with center in $a\in \mathbb{R}^n$ and radius $\varrho>0$, for $1\leq k\leq n$ integer. For $k=n$, we drop the upper-script.
\end{itemize}
\noindent\textit{AI use Statement.} All mathematical content of the paper is solely due to the author. LLMs were used only 
at a later stage of manuscript preparation, mainly for reference matching and for checking certain passages of the presented proofs.
\section{Preliminaries}
This section follows closely the section with similar name in \cite{BCCLXZ} and has the goal of introducing the reader to the definitions and incidence results used in this paper.
\begin{definition}
A $\delta$-tube is the $\delta$-neighborhood of a unit line segment in $\mathbb{R}^4$. A $\delta$-cube is a set of the form $Q = [0, \delta)^4+v$, where $v\in(\delta\mathbb{Z})^4.$
\end{definition}
\begin{definition}[Shading]
Let $T$ be a $\delta$-tube. A shading of $T$ is a set $Y(T)$ that is a union of $\delta$-cubes contained in $T$. Let $\mathbb{T}$ be a family of $\delta$-tubes and
for each $T\in\mathbb{T}$, let $Y(T)$ be a shading of $T$. We denote the pair $(\mathbb{T},Y)$ as a family of tubes and their associated shading.
\end{definition}
\begin{definition}[Refinement]
Let $(\mathbb{T},Y)$ be a set of tubes and their associated shading.
We say that a pair $(\mathbb{T}',Y')$ is a $t$-refinement of $(\mathbb{T},Y)$ if $T'\subset T, Y'(T) \subset Y(T)$, for each $T\in \mathbb{T}$, and 
$$\sum_{T\in\mathbb{T}'} |Y'(T)|\geq t \sum_{T\in\mathbb{T}}|Y(T)|.$$
In the case $t\sim|\log(\delta)|^{-1}$, we use the symbol $\gtrapprox$ to absorb the $t$-contribution.
\end{definition}

Now, it is important to define some essential quantities associated with the pair
$(\mathbb{T},Y)$ in order to have a minimal characterization of its nature.
\begin{definition}
Fix $(\mathbb{T}, Y)$ a set of $\delta$-tubes and their associated shading. We define the sets: $E_Y:=\bigcup_{T\in\mathbb{T}}Y(T)$ and for each $\delta$-cube $Q$,
$\mathbb{T}_Y(Q)
:=
\{T\in\mathbb{T}:Q\subset Y(T)\}$. Moreover, let
$\mathcal{Q}(Y)$ be the set of $\delta$-cubes that are contained in at least one shading $Y(T)$, for some tube $T\in\mathbb{T}$. 
\end{definition}
\begin{definition}
We define the average shading density and average multiplicity of the family of
tubes with shading $(\mathbb{T},Y)$ by, respectively,
$$
\lambda_Y
:=
\frac{1}{\#\mathbb{T}}
\sum_{T\in\mathbb{T}}
\frac{|Y(T)|}{|T|},
$$
and
$$
\mu_Y
:=
\frac{1}{\#\mathcal{Q}_Y}
\sum_{Q\in \mathcal{Q}_Y}
\#\mathbb{T}_Y(Q).
$$
We also say that $(\mathbb{T},Y)$ is $\lambda_Y$-dense if $\frac{|Y(T)|}{|T|}\sim \lambda_Y$, for every $T\in\mathbb{T}$.
\end{definition}

One main object of interest is studying families of tubes with associated shading that have a particular kind of structure, which was proposed in \cite{GZ18} and \cite{KZ}. 

\begin{definition}[Plany]
Let $(\mathbb{T},Y)$ be a set of $\delta$-tubes and their associated shading. We say that $(\mathbb{T},Y)$ is plany if for each $\delta$-cube $Q$,
there is a plane $\Pi(Q)$ so that for all $T\in\mathbb{T}_Y(Q)$ we have $\angle(V(T),\Pi(Q))\lesssim\delta$,
where $V(T)$ is the direction of $T$.
\end{definition}

For prescribing families of tubes whose shading does not concentrate ``too much'', we define the so called two-ends condition also introduced in \cite{KZ}.

\begin{definition}[Two-ends condition]
Let $(\mathbb{T}, Y)$ be a set of $\delta$-tubes and their associated shading. We say that $(\mathbb{T}, Y)$ is $(\varepsilon_1,C_1)$-two-ends if, for all $T \in \mathbb{T}$ and all $\delta \le r \le 1$, we have
\[
\#
\left\{
Q : Q \subset Y(T) \cap B(x,r)
\right\}
\leq C_1
r^{\varepsilon_1} \lambda_Y \delta^{-1}, \,\, \text{for all balls } B(x,r).
\]
\end{definition}

Intuitively, the above definition corresponds to families in which the associated shading of each tube cannot concentrate at a large amount on balls. Note that this property together with $\lambda$-dense shading implies Definition~$1.21$ in \cite{WW24} (for instance, take $B(x,r)$ containing the tube $J$ and fix $r=\delta^{\varepsilon_1}$), which represents a weaker version of the two-ends definition. This remark is important due to the fact that allows us to use results from both sources that ask for families of tubes satisfying one of the possible two-ends condition.
\begin{remark}
\label{epsilon1remark}
Throughout the paper, whenever a final loss $\delta^{-\varepsilon}$ is allowed, the exponent in the two-ends condition will be chosen sufficiently small in terms of $\varepsilon$. More precisely, if a shading satisfies the $(\varepsilon_{1},C_{1})$-two-ends condition, then it also satisfies the $(\varepsilon_{1}^*,C_{1})$-two-ends condition, for every
$0<\varepsilon^*_1\leq \varepsilon_{1}$. We may therefore take $\varepsilon^*_1>0$ sufficiently small relative to the prescribed $\varepsilon>0$. Note that this is only useful for $C_1$ independent of $\delta$. When $C_1$ depends on $\delta$, we instead keep track explicitly of the additional factor $C_1^{-\frac{C}{\varepsilon_1}}$.
\end{remark}

For a condition over the directions of the tubes, we define the following:

\begin{definition}[Robust transversality]
Let $\delta\in(0,1)$ and let $(\mathbb{T},Y)$ be a $\delta$-separated tube family with associated shading. Let $0<\varepsilon_2<1$. We say $Y$ is $(\varepsilon_2,C_2)$-robustly transverse if for all $\delta$-cubes $Q$, all vectors
$W$, and any $\delta<r<1$,
\[
\#\bigl\{
T\in\mathbb{T}_Y(Q):
\angle(V(T),W)\le r
\bigr\}
\leq C_2
r^{\varepsilon_2}\mu_Y.
\]
\end{definition}
Moreover, we can describe families of tubes with a bounded number of elements with ``similar" direction.
\begin{definition}[$m$-parallel]
    A family of $\delta$-tubes $\mathbb T$ is $m$-parallel if, for every $\delta$-ball $B\subset \mathbb{S}^{3}$, there are at most $\lesssim m$ tubes $T\in\mathbb T$ whose directions lie in $B$.
\end{definition}

We will use a two-ends hairbrush estimate derived from the work of Wang--Wu and Gao--Wu--Xi \cite{WW24,GWX25}. Its proof combines the planar two-ends Furstenberg estimate of Wang--Wu developed in \cite{wang2025twoends} with Wolff's hairbrush argument of \cite{Wol95}. Moreover, we prove a stronger estimate in the case where the family has ``few angular concentration". We remark that these estimates were already proven in Lemma $4.1$ in \cite{GWX25} but for our purposes a dependence on the factor $C_1^{\frac{C}{\varepsilon_1}}$ needs to be tracked which asks for additional justification rather than just citing the literature. Before establishing both results, we prove a quantitative planar localization argument.

\begin{lemma}[Quantitative planar two-ends estimate]
\label{quant-planar}
Let $(\mathbb T,Y)$ be a $1$-parallel family of $\delta$-tubes in $\mathbb R^2$, with $(\varepsilon_1,C_1)$-two-ends,
$\lambda_Y$-dense shading, where
$0<\varepsilon_1\leq 1$ and $C_1\geq1$. Assume that, for every $T\in\mathbb T$, $|Y(T)|\geq\lambda_Y\delta$,
and, for every longitudinal interval
$I$ of length $\delta\leq u\leq1$,
\[
|Y(T)\cap I_\delta|
\leq C_1u^{\varepsilon_1}|Y(T)|,
\]
where $I_\delta$ denotes the
$\delta$-neighborhood of $I$. Then, for every $\varepsilon>0$,
\begin{equation}
\label{quant-planar-bound}
|E_Y|
\gtrsim_{\varepsilon,\varepsilon_1}
C_1^{-\frac{C}{\varepsilon_1}}
\delta^\varepsilon
\lambda_Y^{\frac12}
\sum_{T\in\mathbb T}|Y(T)|,
\end{equation}
where $C>0$ is an absolute constant,
independent of $\varepsilon$,
$\varepsilon_1$, $C_1$ and $\delta$.
\end{lemma}

\begin{proof}
We use the standard two-ends reduction that was originated on the work of Wolff. See, for instance, Lemma $6$ of \cite{Tao2009TwoEnds} and Definition $2.19$ of \cite{wang2025twoends}. More precisely, we select a longitudinal
interval maximizing the normalized shading concentration and rescale the corresponding portion of each tube. This produces a two-ends condition with an absolute constant,
while the original $(\varepsilon_1,C_1)$-two-ends
hypothesis provides a lower-bound for the localization scale. We keep track of this lower-bound explicitly in order to obtain the factor $C_1^{-\frac{C}{\varepsilon_1}}$.

First, we shall reduce the statement to the planar two-ends Furstenberg estimate of Wang--Wu. For every $T\in\mathbb T$, define
\[
\nu_T(I)
=
\frac{|Y(T)\cap I_\delta|}{|Y(T)|}.
\]
Choose a dyadic longitudinal interval $I_T$ of length $r_T$ maximizing
\[
M_T=
\max_{\substack{I\text{ dyadic}\\
\delta\leq |I|\leq1}}
\frac{\nu_T(I)}{|I|^{\frac{\varepsilon_1}{2}}}.
\]
Since the whole unit interval contains the shading of $T$, we have
$M_T\geq1$. On the other hand, the two-ends condition gives
$M_Tr_T^{\frac{\varepsilon_1}{2}} = \nu_T(I_T) \leq C_1r_T^{\varepsilon_1}$. Consequently,
\begin{equation}
\label{minimal-planar-scale}
r_T\geq
(CC_1)^{-\frac{2}{\varepsilon_1}},
\end{equation}
where $C\geq 1$ is an absolute constant resulting from a covering of the interval over balls. Define the localized shading $Y_{\mathrm{loc}}(T) = Y(T)\cap(I_T)_\delta$. By maximality of $I_T$, for every
longitudinal interval $J$ of length
$u\leq r_T$, we have $\nu_T(J)\leq M_Tu^{\frac{\varepsilon_1}{2}}$. 
\newpage
Here, arbitrary intervals are covered by a bounded number of dyadic intervals of comparable length. Since
$\nu_T(I_T)=M_Tr_T^{\frac{\varepsilon_1}{2}}$, it follows that
\begin{equation}
\label{localized-two-ends}
\frac{
|Y_{\mathrm{loc}}(T)\cap J_\delta|
}{
|Y_{\mathrm{loc}}(T)|
}
\lesssim
\left(\frac{u}{r_T}\right)^{\frac{\varepsilon_1}{2}}.
\end{equation}
Furthermore, 
\begin{equation}
\label{localized-density}
|Y_{\mathrm{loc}}(T)|
=M_Tr_T^{\frac{\varepsilon_1}{2}}|Y(T)|\geq r_T^{\frac{\varepsilon_1}{2}}|Y(T)|.
\end{equation}
Thus, the localization retains a proportion $\gtrsim r_T^{\frac{\varepsilon_1}{2}}$ of the original shading, while producing a two-ends condition with an absolute constant at scale $r_T$. There are $O\left(\log(\frac{1}{\delta})\right)$ possible
dyadic values of $r_T$. Therefore,
there exists a scale $r$ and a
subfamily $\mathbb T_r\subseteq\mathbb T$
such that $r_T\sim r$ for every
$T\in\mathbb T_r$, and
\begin{equation}
\label{planar-pigeonhole}
\sum_{T\in\mathbb T_r}|Y(T)|
\gtrsim
\frac{1}{\log\left(\frac{1}{\delta}\right)}
\sum_{T\in\mathbb T}|Y(T)|.
\end{equation}
Partition the plane into disjoint
squares $B$ of side length comparable to $r$. Since each localized shading is supported in a longitudinal interval of length comparable to $r$, it meets only $O(1)$ such squares. Assign each $T\in\mathbb T_r$ to a square $B$ that retains a fixed proportion of its localized shading. Denote the resulting shading by $Y_B(T)$, and let $\mathbb T(B)$ be the family of tubes assigned to $B$. By \eqref{localized-density},
\begin{equation}
\label{square-density}
|Y_B(T)|\geq r^{\frac{\varepsilon_1}{2}}|Y(T)|\gtrsim\lambda_Y\delta r^{\frac{\varepsilon_1}{2}}\geq\lambda_Y\delta r,
\end{equation}
where the last inequality follows
from $0<\varepsilon_1\leq1$. Rescale each square $B$ to unit size. Under this rescaling, the localized tubes, that is, $T\cap B$ have width $\frac{\delta}{r}$ and \eqref{square-density} implies
that their rescaled shadings are
$\gtrsim\lambda_Y$-dense. Furthermore, \eqref{localized-two-ends} shows that the rescaled shadings satisfy a two-ends condition with exponent $\frac{\varepsilon_1}{2}$ and an absolute constant independent of $C_1$. Call the isotropic rescaled family $(\widetilde{\mathbb{T}}(B),\widetilde{Y}_B)$. Indeed, for every longitudinal interval of rescaled length $u$,
\begin{equation}
\label{rescaled}
\frac{
|\widetilde{Y}_B(T)\cap J_{\frac{\delta}{r}}|
}{
|\widetilde{Y}_B(T)|
}
\lesssim u^{\frac{\varepsilon_1}{2}}.
\end{equation}

The original tubes have
$\delta$-separated directions.
After rescaling, the relevant
angular separation is $\frac{\delta}{r}$. We partition the tubes assigned to each square $B$ into $O(r^{-1})$ subfamilies with
$\frac{\delta}{r}$-separated directions. Retaining a subfamily carrying a maximal proportion of the shading mass loses at most a factor $\gtrsim r$. Choose parameters $0<\tau<\frac{\sigma\varepsilon_1}{2}$, where $\sigma>0$ is sufficiently small relative to $\varepsilon$. Since the rescaled shading satisfies \eqref{rescaled}, we obtain at scale $u=\left(\frac{\delta}{r}\right)^\sigma$
\[
\frac{
|\widetilde{Y}_B(T)\cap
J_{\frac{\delta}{r}}|
}{
|\widetilde{Y}_B(T)|
}
\lesssim
\left(\frac{\delta}{r}\right)^{\frac{\sigma\varepsilon_1}{2}}
\leq\left(\frac{\delta}{r}\right)^\tau,
\]
for sufficiently small
$\frac{\delta}{r}$. Note that we are allowed to only consider $\frac{\delta}{r}$ sufficiently small since otherwise the desired estimate follows from the trivial multiplicity bound $|E_Y|\gtrsim\delta\sum_{T\in\mathbb{T}}|Y(T)|$ and \eqref{minimal-planar-scale}. Therefore, the one-scale two-ends hypothesis of Wang--Wu's theorem is satisfied. Applying Theorem $2.1$ of \cite{WW24}, choosing the auxiliary parameters sufficiently small, and scaling to the original coordinates, gives
\[
\left|
\bigcup_{T\in\mathbb T(B)}Y_B(T)
\right|
\gtrsim_{\varepsilon,\varepsilon_1}
\left(\frac{\delta}{r}\right)^{\frac{\varepsilon}{4}}
\lambda_Y^{\frac{1}{2}}r
\sum_{T\in\mathbb T(B)}|Y_B(T)|.
\]
Here, the factor $r$ accounts for
the separation into directional
subfamilies. Since the squares are disjoint,
\[
|E_Y|\geq \sum_B \left|
\bigcup_{T\in\mathbb T(B)}Y_B(T)
\right|.\]
Applying the preceding estimate
and then \eqref{square-density},
we obtain
\[
|E_Y|\gtrsim_{\varepsilon,\varepsilon_1}
\left(\frac{\delta}{r}\right)^{\frac{\varepsilon}{4}}
\lambda_Y^{\frac{1}{2}}r^{1+\frac{\varepsilon_1}{2}}
\sum_{T\in\mathbb T_r}|Y(T)|\gtrsim_{\varepsilon,\varepsilon_1}
\frac{
\delta^{\frac{\varepsilon}{4}}r^{1+\frac{\varepsilon_1}{2}}
}{
\log\left(\frac{1}{\delta}\right)
}
\lambda_Y^{\frac{1}{2}}
\sum_{T\in\mathbb T}|Y(T)|.
\]
Absorb the logarithmic factor
into an additional arbitrarily
small power of $\delta$. Finally, by
\eqref{minimal-planar-scale} and using $C_1\geq 1$, gives $r^{1+\frac{\varepsilon_1}{2}}
\geq
(CC_1)^{-\frac{2+\varepsilon_1}{\varepsilon_1}}\geq (CC_1)^{-\frac{3}{\varepsilon_1}}$. Consequently, after absorbing
the absolute constants in the
implicit constant and the logarithmic factor in $\delta^{\varepsilon}$, we obtain the exact relation expressed in \eqref{quant-planar-bound}.
\end{proof}

\begin{lemma}[Two-ends hairbrush estimates]
\label{hairbrushest}
Let $(\mathbb T,Y)$ be an $m$-parallel
family of $\delta$-tubes in $\mathbb R^4$
with $(\varepsilon_1,C_1)$-two-ends, where $0<\varepsilon_1\le1$ and $C_1\ge1$. Set $X=\delta^3\#\mathbb T$. Then, for every $\varepsilon>0$,
\begin{equation}
\label{general-hairbrush}
\mu_Y
\lesssim_{\varepsilon,\varepsilon_1}
C_1^{\frac{C}{\varepsilon_1}}
\delta^{-\varepsilon}
m^{\frac23}\lambda_Y^{-\frac34}
\delta^{-1}X^{\frac13}.
\end{equation}
Assume additionally that $m=1$ and
$(\mathbb T,Y)$ satisfies the following angular non-concentration property:
\begin{equation}
\label{angular}
\sup_{Q,W}
\#\left\{
T\in\mathbb T_Y(Q):
\angle(V(T),W)\le\delta^\alpha
\right\}
\le\frac{\mu_Y}{10},
\end{equation}
where $\alpha>0$ is sufficiently small relative to $\varepsilon$. Then,
\begin{equation}
\label{transverse-hairbrush}
\mu_Y
\lesssim_{\varepsilon,\varepsilon_1}
C_1^{\frac{C}{\varepsilon_1}}
\delta^{-\varepsilon}
\lambda_Y^{-\frac34}
\delta^{-1}X^{\frac12}.
\end{equation}
In both estimates, $C>0$ is an absolute
constant independent of $\varepsilon$,
$\varepsilon_1$, $C_1$ and $\delta$.
\end{lemma}

\begin{proof}
The first estimate follows from Lemma~4.1 of \cite{GWX25},
specialized to dimension four, together with the quantitative
two-ends reduction. Indeed, the $m$-parallel hypothesis implies
Wolff's axiom with error $O(m)$, since any $1\times s_1\times s_2\times s_3$ rectangular box contains at most $O(ms_1s_2s_3\delta^{-3})$ tubes. By dyadic pigeonholing, we
first reduce to pointwise comparable shading densities, losing
only logarithmic factors. We then apply the maximal-concentration localization argument
from Lemma~\ref{quant-planar}, which is dimension-independent.
After selecting a common dyadic scale $r$, we have $r\gtrsim_{\varepsilon_1}
(C_1\log(\frac{1}{\delta}))^{-\frac{2}{\varepsilon_1}}$.
The localized shadings satisfy an absolute two-ends condition,
and after rescaling the corresponding $r$-cubes to unit size, their width, density, and Wolff's axiom parameter satisfy
$ \lambda'\gtrsim\lambda r^{\frac{\varepsilon_1}{2}-1}$ and $m'\lesssim mr^{-3}$, where
$\lambda$ is the density after the initial dyadic uniformization.

The localized two-ends estimate \eqref{localized-two-ends} holds at every scale, allowing the auxiliary two-ends parameters in Lemma $4.1$  of \cite{GWX25} to be chosen sufficiently small in terms of an arbitrarily prescribed $\eta>0$, without requiring the original exponent $\varepsilon_1$ to be small. Applying that result at scale $\frac{\delta}{r}$ and scaling back gives
\[
|E_B|\gtrsim
\left(\frac{\delta}{r}\right)^\eta
m^{-\frac{2}{3}}\lambda^{\frac{7}{4}}
r^{\frac{5}{4}+\frac{7\varepsilon_1}{8}} \delta
\bigl(\delta^3\#\mathbb T(B)\bigr)^{\frac{2}{3}}.
\]
Summing over the disjoint localization cubes and using
\[
\sum_B(\#\mathbb T(B))^{\frac{2}{3}}
\geq
\left(\sum_B\#\mathbb T(B)\right)^{\frac{2}{3}},
\]
we obtain the required estimate. Finally, the lower-bound for $r$ yields a factor $C_1^{-\frac{C}{\varepsilon_1}}$, and all logarithmic losses are absorbed into $\delta^\varepsilon$. Thus,
\[
\mu_Y\lesssim_{\varepsilon,\varepsilon_1} C_1^{\frac{C}{\varepsilon_1}}
\delta^{-\varepsilon} m^{\frac{2}{3}}\lambda_Y^{-\frac{3}{4}}
\delta^{-1}X^{\frac{1}{3}}.
\]

We move towards the second result. We first pass to pointwise-dense shadings, since $\lambda_Y$ denotes the average shading density. Set
\[
\mathbb T_*=
\left\{
T\in\mathbb T:
|Y(T)|\ge\frac{\lambda_Y}{2}|T|
\right\},
\]
and let $Y_*$ be the restriction of $Y$ to $\mathbb T_*$. Write
$\lambda_*=\lambda_{Y_*}$, $\mu_*=\mu_{Y_*}$ and $X_*=\delta^3\#\mathbb T_*$. By the definition of average density,
\[
\sum_{T\in\mathbb T_*}|Y(T)|
\ge\frac12\sum_{T\in\mathbb T}|Y(T)|.
\]
Since $\mathcal{Q}_{Y_*}\subset \mathcal{Q}_Y$, it follows that $\mu_*\ge\frac{\mu_Y}{2}$, $\lambda_*\ge\frac{\lambda_Y}{2}$, and $X_*\le X$. Moreover, every remaining tube satisfies
\begin{equation}
\label{essential-density}
|Y_*(T)|\gtrsim\lambda_Y\delta^3.
\end{equation}
The refined shadings retain the two-ends condition, with exponent $\varepsilon_1$ and constant $O(C_1)$. Indeed,
\[
|Y_*(T)\cap B(x,r)|\lesssim C_1r^{\varepsilon_1}
\lambda_Y\delta^3\lesssim C_1r^{\varepsilon_1}|Y_*(T)|.
\]
Furthermore, the angular non-concentration condition is preserved in the form
\[
\sup_{Q,W}
\#\left\{
T\in\mathbb T_{Y_*}(Q):
\angle(V(T),W)\le\delta^\alpha
\right\}
\le\frac{\mu_*}{5}.
\]
Using the definition of average
multiplicity, the Cauchy--Schwarz
inequality gives
\[
\sum_{Q\in\mathcal Q_{Y_*}}
\left(\#\mathbb T_{Y_*}(Q)\right)^2\ge
\frac{
\left(
\sum_{Q\in\mathcal Q_{Y_*}}
\#\mathbb T_{Y_*}(Q)
\right)^2
}{
\#\mathcal Q_{Y_*}
}=
\mu_*
\sum_{Q\in\mathcal Q_{Y_*}}
\#\mathbb T_{Y_*}(Q).
\]
For each $T\in\mathbb T_*$, define
\[
A(T)=
\#\left\{
(Q,T'):
Q\subset Y_*(T)\cap Y_*(T'),\
\angle(V(T),V(T'))>\delta^\alpha
\right\},
\]
where $0<\alpha\ll\varepsilon$. Thus, $\sum_{T\in\mathbb T_*}A(T)$
counts all ordered transverse tube
pairs $(T,T')$, together with a cube
where both are shaded. By the angular non-concentration
property,
\[
\begin{aligned}
\sum_{T\in\mathbb T_*}A(T)
&=
\sum_{Q\in\mathcal Q_{Y_*}}
\#\left\{
(T,T')\in\mathbb T_{Y_*}(Q)^2:
\angle(V(T),V(T'))>\delta^\alpha
\right\}\\
&\ge
\sum_{Q\in\mathcal Q_{Y_*}}
\left(
\left(\#\mathbb T_{Y_*}(Q)\right)^2
-\frac{\mu_*}{5}
\#\mathbb T_{Y_*}(Q)
\right)\gtrsim
\mu_*\lambda_*\delta^{-1}
\#\mathbb T_*.
\end{aligned}
\]
Consequently, there is a tube
$T_0\in\mathbb T_*$ such that
$A(T_0)\gtrsim
\mu_*\lambda_*\delta^{-1}$. Define its transverse hairbrush by
\[
\mathcal H(T_0)=
\left\{
T\in\mathbb T_*:
Y_*(T)\cap Y_*(T_0)\ne\varnothing,\
\angle(V(T),V(T_0))>\delta^\alpha
\right\}.
\]
If $T$ makes angle at least
$\delta^\alpha$ with $T_0$, then
$T$ and $T_0$ meet in at most
$O(\delta^{-\alpha})$ shaded
$\delta$-cubes. It follows that
$\#\mathcal H(T_0)
\gtrsim
\delta^\alpha
\mu_*\lambda_*\delta^{-1}$. Next, we remove a thin neighborhood
of $T_0$, keeping track of the
dependence on $C_1$. Choose $\kappa$ so that $0<\alpha<\kappa\ll\varepsilon$,
and define
$R=(A_0C_1)^{-\frac{1}{\varepsilon_1}}
\delta^\kappa$,
where $A_0$ is a sufficiently large
absolute constant. We first assume that $R\ge c\delta$, where $c>0$ is a sufficiently large geometric constant. For each $T\in\mathcal H(T_0)$, define $Y_{\mathrm{ess}}(T)
= Y_*(T)\setminus N_R(T_0)$. Because
$\angle(V(T),V(T_0))>\delta^\alpha$,
the part of $T$ meeting $N_R(T_0)$
is contained in a longitudinal
segment of length $O(R\delta^{-\alpha})$. By the $(\varepsilon_1,C_1)$-two-ends
condition,
\[
|Y_*(T)\setminus Y_{\mathrm{ess}}(T)|\lesssim
C_1(R\delta^{-\alpha})^{\varepsilon_1}
\lambda_Y\delta^3\lesssim
A_0^{-1}
\delta^{\varepsilon_1(\kappa-\alpha)}
\lambda_Y\delta^3.
\]
Taking $A_0$ sufficiently large,
we obtain 
$|Y_*(T)\setminus Y_{\mathrm{ess}}(T)| \le\frac14\lambda_Y\delta^3.$
Since
$|Y_*(T)|\ge\frac12\lambda_Y|T|$,
we conclude that
$|Y_{\mathrm{ess}}(T)|
\gtrsim\lambda_Y\delta^3,$
for every $T\in\mathcal H(T_0)$. Moreover, the essential shadings
retain the required two-ends
property, with the same exponent
$\varepsilon_1$ and constant
$O(C_1)$. Indeed, for every
$\delta\le u\le1$,
\[
|Y_{\mathrm{ess}}(T)\cap B(x,u)|\le |Y_*(T)\cap B(x,u)|\lesssim
C_1u^{\varepsilon_1}
\lambda_Y\delta^3\lesssim
C_1u^{\varepsilon_1}
|Y_{\mathrm{ess}}(T)|.
\]

We now follow the plane-slicing hairbrush argument presented in Equations $(4.5)$ and $(4.6)$ of \cite{GWX25}, keeping track of the dependence on the two-ends constant $C_1$. Since $m=1$, the random sampling step used in \cite{GWX25} is unnecessary. Instead, we apply the quantitative planar two-ends estimate of Lemma~\ref{quant-planar} to each planar family. Assign the tubes of $\mathcal H(T_0)$ to a $\delta$-separated collection $\mathcal P$ of two-planes through the core line of $T_0$, so that each assigned tube lies in a constant multiple of the
$\delta$-neighborhood of its plane. For each $P\in\mathcal P$, let
$\mathcal H(P)$ denote its assigned
hairbrush tubes. Outside $N_R(T_0)$, the corresponding
plane neighborhoods have overlap
at most $O(R^{-2})$. Consequently,
\begin{equation}
\label{hairbrush-plane-overlap}
|E_{Y_*}|
\gtrsim
R^2
\sum_{P\in\mathcal P}
\left|
\bigcup_{T\in\mathcal H(P)}
Y_{\mathrm{ess}}(T)
\right|.
\end{equation}

We now apply the quantitative planar
two-ends Furstenberg estimate of
Lemma~\ref{quant-planar}. Fix $P\in\mathcal P$ and let $\pi_P$ denote the orthogonal projection
onto $P$. Define the planar
shadings by $Z_P(T)
=
N_{C\delta}
\bigl(\pi_P(Y_{\mathrm{ess}}(T))\bigr)
\cap P$.
Since $Y_{\mathrm{ess}}(T)$ lies in an $O(\delta)$-neighborhood of $P$ and $\pi_P(Y_{\mathrm{ess}}(T))\subset Z_P(T)$, Fubini's theorem gives $|Y_{\mathrm{ess}}(T)|\lesssim\delta|Z_P(T)|$. Combined with \eqref{essential-density}, this yields $|Z_P(T)|\gtrsim \lambda_Y\delta$. Since $T$ lies in an $O(\delta)$-neighborhood of $P$, its direction makes angle $O(\delta)$ with $P$, so projection onto $P$ moves directions by $O(\delta)$. As $m=1$, each $\delta$-arc of directions in $P$ receives $O(1)$ projected directions, so $\mathcal{H}(P)$ splits into $O(1)$ subfamilies with $\delta$-separated projected directions. First, removing $N_R(T_0)$ costs at most half of each shading: by the angle condition, $Y_*(T)\cap N_R(T_0)$ lies in a ball of radius $r\lesssim R\delta^{-\alpha}$, so by the two-ends condition and the choice of $R$,
$|Y_*(T)\setminus Y_{\mathrm{ess}}(T)|\leq C_1r^{\varepsilon_1}|Y_*(T)|\leq \frac{1}{2}|Y_*(T)|.$
Now, let $J\subset P$ be a segment of length $\delta\leq u\leq 1$. Every $z\in Z_P(T)\cap J_\delta$ lies within $O(\delta)$ of $\pi_P(Q)$ for some $\delta$-cube $Q\subset Y_*(T)$ with $\pi_P(Q)\cap N_{O(\delta)}(J)\ne \varnothing$. Call the set of all these cubes by $\mathcal{Q}'$. Because the direction of $T$ makes angle $O(\delta)$ with $P$, $\pi_P$ shortens longitudinal distances in $T$ by at most a constant factor, so all such cubes lie in a ball $B$ of radius $O(u)$. Each $N_{C\delta}(\pi_P(Q))\cap P$ has area $O(\delta^2)=O(\delta^{-1}|Q|)$, hence 
\begin{align*}
|Z_P(T)\cap J_\delta|\leq \sum_{Q\in\mathcal{Q}'} \left|N_{(C+1)\delta}(\pi_P(Q))\cap P\right|\lesssim \delta^{-1} |Y_*(T)\cap B|\lesssim C_1 u^{\varepsilon_1}|Z_P(T)|.
\end{align*}
This shows that the two-ends condition also passes to the projected shadings.
\newpage
Therefore, Lemma~\ref{quant-planar}
applies, with two-ends constant
$O(C_1)$. For every $\eta>0$,
\[
\left|
\bigcup_{T\in\mathcal H(P)}
Z_P(T)
\right|\gtrsim_{\eta,\varepsilon_1}
C_1^{-\frac{C_0}{\varepsilon_1}}
\delta^\eta
\lambda_Y^{\frac{1}{2}}
\sum_{T\in\mathcal H(P)}
|Z_P(T)|\gtrsim_{\eta,\varepsilon_1}
C_1^{-\frac{C_0}{\varepsilon_1}}
\delta^\eta
\lambda_Y^{\frac{3}{2}}\delta
\#\mathcal H(P),
\]
where $C_0>0$ is an absolute constant, independent of $\eta$. Passing from planar area to
four-dimensional volume introduces
a factor $\delta^2$, corresponding
to the two directions normal to $P$.
Thus,
\begin{equation}
\label{planar-hairbrush}
\left|
\bigcup_{T\in\mathcal H(P)}
Y_{\mathrm{ess}}(T)
\right|
\gtrsim_{\eta,\varepsilon_1}
C_1^{-\frac{C_0}{\varepsilon_1}}
\delta^\eta
\lambda_Y^{\frac{3}{2}}\delta^3
\#\mathcal H(P).
\end{equation}
Summing over the planes and using
\eqref{hairbrush-plane-overlap},
we obtain
\[
|E_{Y_*}|
\gtrsim_{\eta,\varepsilon_1}
R^2C_1^{-\frac{C_0}{\varepsilon_1}}
\delta^\eta
\lambda_Y^{\frac{3}{2}}\delta^3
\sum_{P\in\mathcal P}
\#\mathcal H(P)\gtrsim_{\eta,\varepsilon_1}
C_1^{-\frac{C_0+2}{\varepsilon_1}}
\delta^{2\kappa+\eta}
\lambda_Y^{\frac{3}{2}}\delta^3
\#\mathcal H(T_0).
\]
Substituting the lower-bound of $\#\mathcal H(T_0)$ and using double-counting gives
\[
\mu_*
\lesssim_{\eta,\varepsilon_1}
C_1^{\frac{C_0+2}{2\varepsilon_1}}
\delta^{-\kappa-\frac{\eta+\alpha}{2}}
\lambda_Y^{-\frac{3}{4}}
\delta^{-1}X_*^{\frac{1}{2}}.
\]
Choose $\alpha,\kappa,\eta>0$
sufficiently small, with
$0<\alpha<\kappa$, so that
$\kappa+\frac{\eta+\alpha}{2}
<\varepsilon$. Since $\mu_Y\le2\mu_*$ and
$X_*\le X$, it follows that
\[
\mu_Y
\lesssim_{\varepsilon,\varepsilon_1}
C_1^{\frac{C_0+2}{2\varepsilon_1}}
\delta^{-\varepsilon}
\lambda_Y^{-\frac{3}{4}}
\delta^{-1}X^{\frac{1}{2}}.
\]

It remains to consider the case
$R<c\delta$. By the definition of $R$, $C_1\gtrsim_{\varepsilon_1}
\delta^{-(1-\kappa)\varepsilon_1}$.
Since $m=1$, we have
$\#\mathbb T\lesssim\delta^{-3}$ and
$\mu_Y\le\#\mathbb T$. Hence, 
\[
\frac{\mu_Y}
{\lambda_Y^{-\frac{3}{4}}\delta^{-1}X^{\frac{1}{2}}}
\lesssim
\delta^{-\frac{1}{2}}(\#\mathbb T)^{\frac{1}{2}}
\lesssim\delta^{-2}.
\]
On the other hand, since $\kappa<\frac{1}{4}$,
$C_1^{\frac{4}{\varepsilon_1}}
\gtrsim_{\varepsilon_1}
\delta^{-4(1-\kappa)}
\ge\delta^{-2}$. Thus, the desired estimate also holds in this case.
\end{proof}

Next, we cite the trilinear Kakeya estimate in $\mathbb{R}^4$
established by Guth and Zahl.

\begin{theorem}[Proposition $3.1$, \cite{GZ18}; Theorem $1.1$, \cite{KatzRogers18}]
\label{trilinearKakeya}
Let $\mathbb{T}$ be a set of $\delta$-tubes in $\mathbb{R}^4$ that point in
$\delta$-separated directions. Then, for any $\varepsilon>0$, there exists a
constant $C_\varepsilon$ such that
$$
\int_{\mathbb{R}^4}
\left(
\sum_{T',T'',T'''\in\mathbb{T}}
\mathds{1}_{T'}\mathds{1}_{T''}\mathds{1}_{T'''}
\left|
V(T')\wedge V(T'')\wedge V(T''')
\right|^{\frac{12}{13}}
\right)^{\frac{13}{27}}
\leq C_\varepsilon
\delta^{-\frac{1}{3}-\varepsilon}
\bigl(\delta^3\#\mathbb{T}\bigr)^{\frac{4}{3}}.
$$
\end{theorem}

A crucial element of our results will be the following technical lemma, which establishes a connection between the plany and the quantitatively transverse cases.
\begin{lemma}[Lemma $33$, \cite{KZ}]
\label{KZconstruction}
Let $(\mathbb{T},Y)$ be a set of $\delta$-tubes and their associated shading.
Suppose that $(\mathbb{T},Y)$ is $(\varepsilon_1,C_1)$-two-ends
and $(\varepsilon_2,C_2)$-robustly transverse. Then, there exists a number $\delta \leq \rho \leq 1$,
a refinement $(\mathbb{T}_1,Y_1)$ of $(\mathbb{T},Y)$, a set
$\Omega \subset \mathbb{S}^3$ of $\rho$-separated points, a set
$(\mathbb{T}_\rho,Y_\rho)$ of essentially distinct $\rho$-tubes and their associated shading, and numbers $D,D_1,D_2 \geq 1$ with
$D_1 \geq D_2$ and $\frac{D_1}{D_2} \leq D$ satisfying the following properties:
\begin{enumerate}[label=(\roman*)]
\item
For each $Q\in \mathcal{Q}(Y_1)$,
there is a plane $\Pi(Q)$ such that
    $$\angle(V(T),\Pi(Q)) \leq \rho,
\text{ for all } T\in\mathbb{T}_{Y_1}(Q);$$
\item
For each $Q\in \mathcal{Q}(Y_1)$,
we have
\begin{equation}
\label{quanttrans}
\#
\left\{
(T',T'',T''')\in \mathbb{T}_{Y_1}(Q)^3 :
\left|V(T')\wedge V(T'')\wedge V(T''')\right|
\gtrsim \rho
\right\}
\gtrapprox
\left(\#\mathbb{T}_{Y_1}(Q)\right)^3;
\end{equation}

\item
Every tube in $\mathbb{T}_\rho$ points in a direction from $\Omega$. For each $\omega\in\Omega$, there are approximately $A$ tubes from
$\mathbb{T}_\rho$ in direction $\omega$;

\item
Each tube $T_\rho\in\mathbb{T}_\rho$ contains
    $\approx \frac{\#\mathbb{T}_1}{\#\mathbb{T}_\rho}$
tubes from $\mathbb{T}_1$, and each tube from $\mathbb{T}_1$ is
contained in exactly one tube from $\mathbb{T}_\rho$;

\item
If $T\in\mathbb{T}_1$, $T_\rho\in\mathbb{T}_\rho$, and
$T\subset T_\rho$, then
    $Y_1(T)\subset Y_\rho(T_\rho)$;

\item
We have $\lambda_{Y_\rho} \gtrapprox\lambda_{Y_1}
    \gtrapprox \lambda_Y$;

\item
$(\mathbb{T}_\rho,Y_\rho)$ is
$(\varepsilon_1,C_3)$-two-ends with $C_3\lessapprox C_1$;

\item
For each $\delta$-cube  $Q\in \mathcal{Q}(Y_1)$,
we have
    $\#\mathbb{T}_{Y_1}(Q)\approx \mu_{Y_1}$.
For each $\rho$-cube
$Q_\rho\in \mathcal{Q}(Y_\rho)$,
we have $\#\mathbb{T}_\rho(Q_\rho)\approx\mu_{Y_\rho}$;

\item
There is a number $\widetilde{\mu}$ such that for each
$T_\rho\in\mathbb{T}_\rho$, $\mu_{Y_{T_\rho}}\approx\widetilde{\mu}$,
where $\mu_{Y_{T_\rho}}$ is the multiplicity of $\delta$-tubes in the
thicker tube $T_\rho$. Moreover, $\mu_{Y_1}\leq \rho^{-1}\widetilde{\mu}$;
\item
For each $Q_\rho\in \mathcal{Q}(Y_\rho)$,
there are planes
    $\Pi_1(Q_\rho),\ldots,\Pi_{D_1}(Q_\rho)$
and collections of $\rho$-tubes
$\mathbb{T}_\rho^1(Q_\rho),\ldots,
    \mathbb{T}_\rho^{D_1}(Q_\rho)
    \subset \mathbb{T}_\rho(Q_\rho)$
such that the $\rho$-tubes in $\mathbb{T}_\rho^i(Q_\rho)$ satisfy $\angle(V(T_\rho),\Pi_i(Q_\rho))\leq\rho$
and $\#\mathbb{T}_\rho^i(Q_\rho)= \frac{\mu_{Y_\rho}}{D}$
for every $1\leq i\leq D_1$.
Each $T\in\mathbb{T}_\rho(Q_\rho)$ is contained in $D_2$
collections of the form $\mathbb{T}_\rho^i(Q_\rho)$;

\item
We have the multiplicity bound
$\mu_{Y_1} \lessapprox\frac{\mu_{Y_\rho} \cdot \,\widetilde{\mu}}{D}$.

\end{enumerate}
\end{lemma}
Furthermore, the above theorem leads to a rather useful result regarding the trilinear branch of the Katz--Zahl plany/trilinear decomposition. After the usual refinement, for each relevant $\delta$-cube $Q$, either the tubes in $\mathbb T_Y(Q)$ are concentrated near a $2$-plane at scale $\rho$, leading to the plany case, or a positive proportion of triples are quantitatively transverse, that is, satisfy the condition in \eqref{quanttrans}. This allows us to establish other useful upper-bound for the average multiplicity of the family in question. The obtained bound will correspond to the $G$-estimate in Section~\ref{Gestimate} and can be stated as follows:
\begin{corollary}[Corollary $2.1$, \cite{BCCLXZ}]
Let $(\mathbb{T}, Y)$ be a family of $1$-parallel tubes such that 
\[
\#
\left\{
T', T'', T''' \in \mathbb{T}_Y(Q) :
\left|V(T') \wedge V(T'') \wedge V(T''')\right|
\gtrsim \rho
\right\}
\gtrapprox \left(\#\mathbb{T}_Y(Q)\right)^3.
\]
Then, the average multiplicity has the following estimate
\[
\mu_Y
\lessapprox
\lambda^{-\frac{9}{4}}
\rho^{-1}
\delta^{-\frac{3}{4}}
\left(\delta^3 \#\mathbb{T}\right)^{\frac{3}{4}}.
\]
\end{corollary}

As a final important remark, it is useful to see that if the family before refinement in the sense of Lemma~\ref{KZconstruction} is robustly transverse, then the fine family satisfies the angular non-concentration property with the appropriate adapted parameters.
\begin{corollary}[Transversality of the fine refinement]
\label{cor:fine-transverse}
Let $(\mathbb T,Y)$ be
$(\varepsilon_2,C_2)$-robustly transverse, and let $(\mathbb T_1,Y_1)$ be the fine refinement supplied by Lemma~\ref{KZconstruction}. Suppose that the latter is a $t$-refinement of $(\mathbb{T},Y)$ for some parameter $t$. Then, for every fine shaded cube $Q$, every direction $W$, and every $\delta<r<1$,
\[
\#\left\{
T\in\mathbb T_{Y_1}(Q):
\angle(V(T),W)\le r
\right\}
\le
C_2t^{-1}r^{\varepsilon_2}\mu_{Y_1}.
\]
In particular, since the refinement is of logarithmic/subpower order, we obtain that the family $(\mathbb{T}_1,Y_1)$ satisfies condition \eqref{angular} for $\alpha>0$ and sufficiently small $\delta$.
\end{corollary}

\begin{proof}
Since $\mathcal{Q}_{Y_1}\subseteq \mathcal{Q}_Y$ and the refined incidence is at least $t$ times the original one,
\begin{equation}
\label{multiplicity}
\mu_{Y_1}
=
\frac{1}{\#\mathcal{Q}_{Y_1}}
\sum_{Q\in \mathcal{Q}_{Y_1}}
\#\mathbb{T}_{Y_1}(Q)
\ge
t\frac{1}{\#\mathcal{Q}_Y}
\sum_{Q\in \mathcal{Q}_Y}
\#\mathbb{T}_Y(Q)
=t\mu_Y.
\end{equation}
Every tube counted on the left-hand
side of \eqref{multiplicity} is also counted in the original family. Therefore, the robust transversality of $(\mathbb T,Y)$ gives
\[
\#\{T\in\mathbb T_{Y_1}(Q):
\angle(V(T),W)\le r\}\le C_2r^{\varepsilon_2}\mu_Y\le C_2t^{-1}r^{\varepsilon_2}\mu_{Y_1}.
\]
Setting $r=\delta^\alpha$ and using the fact that, for sufficiently small values of $\delta$, there exists a parameter $\sigma$ that satisfies $0<\sigma<\alpha\varepsilon_2$ and $t\gtrsim_\sigma\delta^\sigma$, we prove the final assertion.
\end{proof}

\section{Proof of Theorem~\ref{thm:main-restriction} and Sharpness}
We follow a similar procedure to \cite{BCCLXZ}, that is, we first establish a suitable incidence estimate and, using the decoupling ingredients provided in \cite{WW24}, get the desired bound.
\subsection{A first improved two-ends incidence estimate}
In order to get the first meaningful result, we shall introduce some notation that will help on characterizing the incidence estimates for a two-ends family with its associated shading.
\begin{definition}[$\mathbf{TE}(d,a,b)$]
\label{TEdef}
Let $\delta>0$, and let $(\mathbb{T},Y)$ be a set of $\delta$-tubes in
$\mathbb{R}^4$ pointing in $\delta$-separated directions, together with
their associated shading. Suppose that $(\mathbb{T},Y)$ is
$(\varepsilon_1,100)$-two-ends. Then, for each $\varepsilon>0$,
\begin{equation}
\left|
\bigcup_{T\in\mathbb{T}} Y(T)
\right|
\geq
c\,\lambda_Y^a\,
\delta^{4-d+\varepsilon}
\left(\delta^3 \#\mathbb{T}\right)^b,
\end{equation}
where $c=c(d,a,b,\varepsilon,\varepsilon_1)>0$ is a constant independent of $\delta$.
\end{definition}
\begin{definition}[$\mathbf{RT}(d,a,b)$]
Let $\delta>0$, and let $(\mathbb{T},Y)$ be a set of $\delta$-tubes in
$\mathbb{R}^4$ pointing in $\delta$-separated directions, together with
their associated shading. Suppose that $(\mathbb{T},Y)$ is
$(\varepsilon_1,C_1)$-two-ends and
$(\varepsilon_2,100)$-robustly transverse. Then, for each
$\varepsilon>0$,
\begin{equation}
\left|
\bigcup_{T\in\mathbb{T}} Y(T)
\right|
\geq
c\,\lambda_Y^a\,
C_1^{-\frac{C}{\varepsilon_1}}\,
\delta^{4-d+\varepsilon}
\left(\delta^3 \#\mathbb{T}\right)^b,
\end{equation}
where $C=C(d,a,b)$
and $c=c(d,a,b,\varepsilon,\varepsilon_1,\varepsilon_2)>0$ are independent of $\delta$.
\end{definition}
\begin{remark}
\label{C1remark}
Note that we may always assume $C_1\geq 1$, since enlarging the two-ends constant does not affect the hypothesis. Moreover, we will not explicitly track the changes in the constant $C$ while deriving \textbf{RT}-estimates to avoid heavy calculations.
\end{remark}
There is also a relation between those two definitions under certain conditions proved by Katz--Zahl.
\begin{proposition}[Proposition $14$, \cite{KZ}]
\label{RTtoTE}
Suppose that the assertion $\mathbf{RT}(d,a,b)$ is true for some $1\leq d\leq 4,$ $a\geq 0$, 
$\frac{d-1}{3}\leq b\leq 1$. Then, assertion $\mathbf{TE}(d,a,b)$ is true.
\end{proposition}
We prove the following estimate:
\begin{theorem}\label{thm:improved-two-ends}
Let $(\mathbb T,Y)$ be a $(\varepsilon_1,C_1)$-two-ends family of $\delta$-tubes in $\mathbb R^4$, with tubes pointing in $\delta$-separated directions. Then, the following estimate is valid:
\begin{equation}\label{eq:main-union}
\left|
\bigcup_{T\in\mathbb T} Y(T)
\right|
\gtrapprox C_1^{-\frac{C}{\varepsilon_1}}
\lambda^{\frac{179}{89}}
\delta^{\frac{87}{89}}
\bigl(\delta^3 \#\mathbb T\bigr)^{\frac{493}{534}},
\end{equation}
for some constant $C>0$. As a consequence, the assumption 
$\mathbf{TE}
\left(
\frac{269}{89},
\frac{179}{89},
\frac{493}{534}
\right)$
holds.
\end{theorem}
\begin{proof}
Let us first assume that the family of tubes in question is, additionally, $(\varepsilon_2,100)$-robustly transverse. This extra assumption will be removed at the end of the argument. Apply the Katz--Zahl fine/coarse decomposition established in Lemma~\ref{KZconstruction}. We obtain a scale $\delta\leq \rho\leq 1$, a fine family $(\mathbb T_1,Y_1),$ and a coarse family of $\rho$-tubes $(\mathbb T_\rho,Y_\rho).$
Write $\mu_1:=\mu_{Y_1},$ $
\mu_\rho:=\mu_{Y_\rho},$
and define
$$
X_1:=\delta^3\#\mathbb T_1,
\qquad
X_\rho:=\rho^3\#\mathbb T_\rho.
$$
The decomposition supplies the following properties: first, the fine multiplicity is uniform:
$\#\mathbb T_{Y_1}(Q) \approx
\mu_1$, for every surviving shaded fine $\delta$-cube $Q$. Second, in each such $Q$,
\begin{equation}\label{eq:trilinearity}
\#\left\{
(T_1,T_2,T_3)\in
\mathbb T_{Y_1}(Q)^3:
\left|
V(T_1)\wedge V(T_2)\wedge V(T_3)
\right|
\gtrsim \rho
\right\}
\gtrapprox
\mu_1^3.
\end{equation}
Third, each coarse tube contains a uniform number
$n
\approx
\frac{\#\mathbb T_1}{\#\mathbb T_\rho}$
of fine descendants. In terms of $X_1$ and $X_\rho$, it can be expressed as the following
\begin{equation}\label{eq:n-X}
n
\approx
\left(\frac{\rho}{\delta}\right)^3
\frac{X_1}{X_\rho}.
\end{equation}
Fourth,
$\lambda_\rho
\gtrapprox
\lambda_1\gtrapprox \lambda$. Finally, the coarse multiplicity is uniform: $\#\mathbb T_{Y_\rho}(Q_\rho)
\approx \mu_\rho,$ for every active shaded $\rho$-cube $Q_\rho$.

At an intermediate stage of the Katz--Zahl refinement, a quantity $M$ is selected so that, for every active coarse pair $(T_\rho,Q_\rho)$,
$$
\sum_{T\in\mathbb T_3(T_\rho)}
\#\bigl\{
Q:Q\subset Y_3(T)\cap Q_\rho
\bigr\}
\sim M,
$$
where $\mathbb T_3(T_\rho)
=\{T\in\mathbb T_3:\ T\text{ is a fine tube belonging to the coarse tube }T_\rho\}$. Here $M$ is first obtained for the preceding refinement $(\mathbb T_2,Y_2)$, and Katz--Zahl explicitly note that this uniformity continues to hold for $(\mathbb T_3,Y_3)$. Their density relation gives (see Equation $(82)$ of \cite{KZ})
\begin{equation}\label{eq:KZ-density-M}
\lambda_3
\sim
\frac{\delta}{\rho}
M
\frac{\#\mathbb T_\rho}{\#\mathbb T_3}
\lambda_\rho.
\end{equation}
Therefore,
\begin{equation}\label{eq:M-expression}
M
\sim
\frac{\rho}{\delta}
\frac{\#\mathbb T_3}{\#\mathbb T_\rho}
\frac{\lambda_3}{\lambda_\rho}=\frac{\rho}{\delta}
\frac{\#\mathbb T_3}{\#\mathbb T_\rho}q^{-1}\gtrapprox \frac{\rho}{\delta}nq^{-1}, \text{ with } q=\frac{\lambda_{\rho}}{\lambda_3}.
\end{equation}
The final family $(\mathbb T_1,Y_1)$ is a logarithmic/subpower refinement of $(\mathbb T_3,Y_3)$, with
$\lambda_1\approx\lambda_3$ and
$\#\mathbb T_1\leq\#\mathbb T_3$.

Let $\mathcal{R}$ denote the collection of active $\rho$-cubes of the final coarse family. For $Q_\rho\in\mathcal R$, we define
$I_3(Q_\rho)$ as the number of fine tubes with $\delta$-cube incidences inside $Q_\rho$ at the intermediate $M$-uniform stage. There are $\approx\mu_\rho$ active coarse tubes through $Q_\rho$, and each active coarse pair contributes $\sim M$ fine incidences. Consequently,
\begin{equation}\label{eq:I3}
I_3(Q_\rho) \approx \mu_\rho M,
\end{equation}
for every $Q_\rho\in\mathcal R$.
Let $I_1(Q_\rho)$ be the analogous incidence count for the final fine family. Since $(\mathbb T_1,Y_1)$ is a logarithmic/subpower refinement of the intermediate family, it retains a logarithmic/subpower fraction of total shaded incidence. Therefore,
\begin{equation}\label{eq:incidence-retention}
\sum_{Q_\rho\in\mathcal R}
I_1(Q_\rho)=\delta^{-4} \sum_{T\in\mathbb{T}_1} |Y_1(T)| 
\gtrapprox \delta^{-4} \sum_{T\in\mathbb{T}_3} |Y_3(T)| \geq
\sum_{Q_\rho\in\mathcal R}
I_3(Q_\rho).
\end{equation}
Using \eqref{eq:I3},
$$
\frac{1}{\#\mathcal R}
\sum_{Q_\rho\in\mathcal R}
I_1(Q_\rho)
\gtrapprox
\mu_\rho M.
$$
By averaging, there exists at least one active $\rho$-cube $Q_\rho$ satisfying
\begin{equation}\label{eq:good-rho-cube}
I_1(Q_\rho)
\gtrapprox
\mu_\rho M.
\end{equation}
Fix such a $Q_\rho$. Set
$$
E_{Q_\rho}
:=
Q_\rho
\cap
\bigcup_{T\in\mathbb T_1}Y_1(T).
$$
Let
$$
\mathbb T_{Q_\rho}
:=
\left\{
T\in\mathbb T_1:
Y_1(T)\cap Q_\rho\neq\varnothing
\right\},
$$
and write
$N_{Q_\rho}
:=
\#\mathbb T_{Q_\rho}.$ There are $\approx\mu_\rho$ coarse tubes through $Q_\rho$, and each coarse tube contains $\approx n$ fine descendants. Therefore,
$N_{Q_\rho}
\lessapprox
\mu_\rho n.$ By definition and rearranging the sums, one has
$$
I_1(Q_\rho) = \sum_{T_\rho\in \mathbb{T}_\rho}
\sum_{T\in\mathbb T_1(T_\rho)}
\#\{Q:Q\subset Y_1(T)\cap Q_\rho\}=\sum_{\substack{Q\in \mathcal{Q}(Y_1)\\Q\subset Q_\rho}}
\#\mathbb T_{Y_1}(Q).
$$
Using the uniform fine multiplicity,
$$
\delta^{-4}|E_{Q_\rho}|\sim \#\{Q: Q\subset E_{Q_\rho}\}\approx \frac{I_1(Q_\rho)}{\mu_1}.
$$
Hence,
$$
|E_{Q_\rho}|
\approx
\delta^4
\frac{I_1(Q_\rho)}{\mu_1}.
$$
Combining this with \eqref{eq:good-rho-cube} gives the following lower-bound:
\begin{equation}\label{eq:EQ-lower}
|E_{Q_\rho}|
\gtrapprox
\delta^4
\frac{\mu_\rho M}{\mu_1}.
\end{equation}

\subsubsection{The $M$-retained estimate \texorpdfstring{$R$}{R}}
Apply the Guth--Zahl trilinear Kakeya estimate to $\mathbb T_{Q_\rho}$ expressed in Theorem~\ref{trilinearKakeya}:
\begin{equation}\label{eq:GZ-local}
\int_{\mathbb{R}^4}
\left(
\sum_{T',T'',T'''\in\mathbb{T}_{Q_\rho}}
\mathds{1}_{T'}
\mathds{1}_{T''}
\mathds{1}_{T'''}
\left|
V(T')\wedge V(T'')\wedge V(T''')
\right|^{\frac{12}{13}}
\right)^{\frac{13}{27}}
\lessapprox
\delta^{-\frac{1}{3}}
(\delta^3N_{Q_\rho})^{\frac{4}{3}}.
\end{equation}
Consider a fine shaded $\delta$-cube $Q\subset E_{Q_\rho}$. By \eqref{eq:trilinearity}, there are $\gtrapprox\mu_1^3$ triples
$
(T',T'',T''')
\in
\mathbb T_{Y_1}(Q)^3
$
satisfying
$$
\left|
V(T')\wedge V(T'')\wedge V(T''')
\right|
\gtrsim
\rho.
$$
Therefore, throughout $E_{Q_\rho}$, the quantity inside the parentheses in \eqref{eq:GZ-local} is bounded below by
$
\gtrapprox
\mu_1^3\rho^{\frac{12}{13}}.
$
After taking the $\frac{13}{27}$-power, the integrand is bounded below by
$\mu_1^{\frac{13}{9}}\rho^{\frac{4}{9}}$.
Consequently,
$$
|E_{Q_\rho}|
\mu_1^{\frac{13}{9}}
\rho^{\frac{4}{9}}
\lessapprox
\delta^{-\frac{1}{3}}
(\delta^3N_{Q_\rho})^{\frac{4}{3}}.
$$
Insert \eqref{eq:EQ-lower}:
$$
\delta^4
\frac{\mu_\rho M}{\mu_1}
\mu_1^{\frac{13}{9}}
\rho^{\frac{4}{9}}
\lessapprox
\delta^{-\frac{1}{3}}
(\delta^3N_{Q_\rho})^{\frac{4}{3}}.
$$
Since the $\delta^4$-factors cancel, raising to the $\frac{9}{4}$-power yields
$$
\mu_1
\lessapprox
\delta^{-\frac{3}{4}}
\rho^{-1}
N_{Q_\rho}^3
(\mu_\rho M)^{-\frac{9}{4}}.
$$
Using $N_{Q_\rho}\lessapprox
\mu_\rho n,
$
$$
\mu_1
\lessapprox
\delta^{-\frac{3}{4}}
\rho^{-1}
\mu_\rho^{\frac{3}{4}}
n^3M^{-\frac{9}{4}}.
$$
Therefore, by \eqref{eq:n-X} and \eqref{eq:M-expression}:
$$
n^3M^{-\frac{9}{4}}
\lesssim
n^{\frac{3}{4}}
\left(
\frac{\rho}{\delta}
\right)^{-\frac{9}{4}}
q^{\frac{9}{4}}\approx \left(\left(
\frac{\rho}{\delta}
\right)^{\frac{9}{4}}
\left(
\frac{X_1}{X_\rho}
\right)^{\frac{3}{4}}\right)\left(
\frac{\rho}{\delta}
\right)^{-\frac{9}{4}}
q^{\frac{9}{4}}=\left(
\frac{X_1}{X_\rho}
\right)^{\frac{3}{4}}q^{\frac{9}{4}}.
$$
Combining everything gives the new local estimate
\begin{equation}
\label{preR}
\mu_1
\lessapprox
q^{\frac{9}{4}}
\delta^{-\frac{3}{4}}
\rho^{-1}
\mu_\rho^{\frac{3}{4}}
\left(
\frac{X_1}{X_\rho}
\right)^{\frac{3}{4}}.
\end{equation}
Applying the first part of Lemma~\ref{hairbrushest} to the $A$-parallel coarse family
$(\mathbb T_\rho,Y_\rho)$ gives
\begin{equation}
\label{coarse}
\mu_\rho
\lessapprox C_3^{\frac{C}{\varepsilon_1}}
A^{\frac{2}{3}}
\lambda_\rho^{-\frac{3}{4}}
\rho^{-1}
X_\rho^{\frac{1}{3}}.
\end{equation}
Substituting it into \eqref{preR} and using property $(vii)$ of Lemma~\ref{KZconstruction},
$$
\mu_1
\lessapprox C_1^{\frac{C}{\varepsilon_1}}
q^{\frac{9}{4}}
A^{\frac{1}{2}}
\lambda_\rho^{-\frac{9}{16}}
\delta^{-\frac{3}{4}}
\rho^{-\frac{7}{4}}
X_1^{\frac{3}{4}}
X_\rho^{-\frac{1}{2}}.
$$
Since $\mathbb T_1$ is $1$-parallel, each $\rho$-direction class contains at most $O((\frac{\rho}{\delta})^3)$ fine tubes, while property $(iii)$ of Lemma~\ref{KZconstruction} gives $\#\mathbb T_\rho\approx A\#\Omega$. Hence, $\#\mathbb T_1\lesssim(\frac{\rho}{\delta})^3\#\Omega$ and the following relation holds:
\begin{equation}
\label{Arelation}
\frac{X_1}{X_\rho}
\lesssim
A^{-1}.
\end{equation}
Recall that
$
\lambda_\rho
\approx
q\lambda_1.
$
Thus, we obtain
\begin{equation}\label{eq:L}
\mu_1
\lessapprox C_1^{\frac{C}{\varepsilon_1}}
q^{\frac{27}{16}}
\lambda_1^{-\frac{9}{16}}
\delta^{-\frac{3}{4}}
\rho^{-\frac{7}{4}}
X_1^{\frac{1}{4}}=: R.
\end{equation}

\subsubsection{The Planebrush estimate \texorpdfstring{$P$}{P}}

Next, we retain the fine and coarse densities separately in the calculation of \cite{BCCLXZ} leading to their multiplicity estimate. Fix a coarse tube $T_\rho$. After rescaling $T_\rho$ to unit scale, its fine descendants form a family of $\frac{\delta}{\rho}$-tubes. Moreover, the inverse map of the anisotropic rescaling is uniformly Lipschitz, so the two-ends inequality expressed in terms of volumes passes to the induced rescaled shadings, with the appropriate change in the two-ends constant. By property $(ix)$ of Lemma~\ref{KZconstruction}, we get that the multiplicity of this rescaled family is a fixed number $\widetilde\mu$. Before the final density replacement, we look at the geometric planebrush estimate in Equation $(26)$ of \cite{BCCLXZ}. Applying Proposition $34$ of Katz–Zahl to the coarse family $(\mathbb T_\rho,Y_\rho)$ and using properties $(iii)$, $(vii)$, and $(x)$ of Lemma $\ref{KZconstruction}$, we obtain the similar presented result but with an explicit $C_1^{\frac{C}{\varepsilon_1}}$ dependence:
\begin{equation}\label{eq:P-coarse}
\mu_\rho
\lessapprox C_3^{\frac{C}{\varepsilon_1}}
\lambda_\rho^{-\frac{1}{3}}
\rho^{-\frac{2}{3}}
A^{\frac{1}{3}}
D^{\frac{4}{3}}\lessapprox C_1^{\frac{C}{\varepsilon_1}}
\lambda_\rho^{-\frac{1}{3}}
\rho^{-\frac{2}{3}}
A^{\frac{1}{3}}
D^{\frac{4}{3}}.
\end{equation}
Interpolating \eqref{eq:P-coarse} and \eqref{coarse} with respective exponents $\frac{3}{4}$ and $\frac{1}{4}$ yields
$$\mu_\rho\lessapprox C_1^{\frac{C}{\varepsilon_1}}\lambda_\rho^{-\frac{7}{16}}\rho^{-\frac{3}{4}}A^{\frac{5}{12}}X_\rho^{\frac{1}{12}}D.$$
After applying property $(xi)$ of Lemma~\ref{KZconstruction}, one has:
\begin{equation}
\label{mu1andtilde}
\mu_1\lessapprox C_1^{\frac{C}{\varepsilon_1}}\lambda_\rho^{-\frac{7}{16}}\rho^{-\frac{3}{4}}A^{\frac{5}{12}}X_\rho^{\frac{1}{12}} \widetilde{\mu}.
\end{equation}
We may work with cubical shadings at the rescaled scale $\frac{\delta}{\rho}$. Indeed, the anisotropic
rescaling maps the original $\delta$-cubes to
bricks of dimensions comparable to
$\delta\times \frac{\delta}{\rho}\times \frac{\delta}{\rho}\times \frac{\delta}{\rho}$.
Denote the resulting rescaled family by $(\widetilde{\mathbb T},\widetilde Y)$, with shading density $\widetilde\lambda$ and
multiplicity $\widetilde\mu$. Assigning each brick to an $\frac{\delta}{\rho}$-cube containing a fixed proportion of its volume, we dyadically pigeonhole the occupied proportions
and fill the retained cubes. We then retain the tubes whose resulting shading has at least half the average volume, preserving
at least half of the remaining incidence. Writing $K:=1+\left\lceil\log_2(\frac{1}{\rho})\right\rceil$, we obtain a subfamily $(\mathbb T_*,Y_*)$ with full-$\frac{\delta}{\rho}$-cube
shadings satisfying
\begin{equation}
\label{anisotropic}
\widetilde\mu\lesssim K\mu_*,
\qquad
\lambda_*\gtrsim K^{-1}\widetilde\lambda,
\qquad
C_{1}^* \lesssim \max\left\{
1,C_1\frac{\lambda_Y}
{\widetilde{\lambda}}
\right\}K\lessapprox KC_1,
\end{equation}
without increasing the number of tubes. The two-ends bound follows by comparing each filled cube with its original occupied portion, allowing an absolute enlargement
of the tubes. We use this reduction whenever applying
an estimate after anisotropic rescaling. Applying the two-ends hairbrush estimate at scale $\frac{\delta}{\rho}$ gives
\begin{equation}\label{eq:mu-tilde}
\widetilde\mu
\lessapprox C_1^{\frac{C}{\varepsilon_1}}
\lambda_1^{-\frac{3}{4}}
\left(
\frac{\delta}{\rho}
\right)^{-1}
\left(
\frac{X_1}{X_\rho}
\right)^{\frac{1}{3}}.
\end{equation}
By setting $\lambda_1(T_\rho):=
\frac1{\#\mathbb T_1(T_\rho)}
\sum_{T\in\mathbb T_1(T_\rho)}
\frac{|Y_1(T)|}{|T|}$ and using property $(iv)$ of Lemma~\ref{KZconstruction}:
$$\lambda_1=\frac{1}{\#\mathbb{T}_1}\sum_{T_\rho\in\mathbb{T}_\rho}\sum_{T\in\mathbb T_1(T_\rho)}
\frac{|Y_1(T)|}{|T|}=\sum_{T_\rho\in\mathbb{T}_\rho}\frac{\#\mathbb{T}_1(T_\rho)}{\#\mathbb{T}_1}\lambda_1(T_\rho).$$
Since $\sum_{T_\rho\in\mathbb{T}_\rho}\frac{\#\mathbb{T}_1(T_\rho)}{\#\mathbb{T}_1}=1$, one can choose a coarse tube satisfying $\lambda_1(T_\rho)\geq\lambda_1$. Anisotropic rescaling preserves the normalized shading density, up to absolute constants, which gives us $\widetilde{\lambda}\gtrapprox\lambda_1(T_\rho)\geq\lambda_1$. In this context, the number of fine descendants inside one fixed $T_\rho$ on the rescaled family satisfies 
\begin{equation}
\label{descendantsrelation}
\left(\frac{\delta}{\rho}\right)^3
\#\mathbb{T}_{1}(T_\rho)
\approx
\frac{
\delta^3\#\mathbb T_1
}{
\rho^3\#\mathbb T_\rho
}=\frac{X_1}{X_\rho}.
\end{equation}
Substituting \eqref{eq:mu-tilde} into \eqref{mu1andtilde} and using \eqref{Arelation},
$$
\mu_1
\lessapprox C_1^{\frac{C}{\varepsilon_1}}
\lambda_\rho^{-\frac{7}{16}}
\lambda_1^{-\frac{3}{4}}
\rho^{-\frac{3}{4}}
\left(
\frac{\delta}{\rho}
\right)^{-1}
A^{\frac{5}{12}}
X_\rho^{\frac{1}{12}}
\left(
\frac{X_1}{X_\rho}
\right)^{\frac{1}{3}}=  C_1^{\frac{C}{\varepsilon_1}}\lambda_\rho^{-\frac{7}{16}}
\lambda_1^{-\frac{3}{4}}
\rho^{\frac{1}{4}}
\delta^{-1}
X_1^{-\frac{1}{12}}
X_\rho^{\frac{1}{6}}.
$$
Using once again 
$\lambda_\rho\approx q\lambda_1$, we arrive at
\begin{equation}\label{preliminar}
\mu_1
\lessapprox C_1^{\frac{C}{\varepsilon_1}}
q^{-\frac{7}{16}}
\lambda_1^{-\frac{19}{16}}
\rho^{\frac{1}{4}}
\delta^{-1}
X_1^{-\frac{1}{12}}
X_\rho^{\frac{1}{6}}.
\end{equation}
Separately, one can use property $(ix)$ of Lemma~\ref{KZconstruction} over \eqref{eq:mu-tilde} obtaining
\begin{equation}
\label{preliminar1}
\mu_1\lessapprox C_1^{\frac{C}{\varepsilon_1}}\lambda_1^{-\frac{3}{4}}\delta^{-1}X_1^{\frac{1}{3}}X_\rho^{-\frac{1}{3}}.
\end{equation}
By interpolating \eqref{preliminar} with \eqref{preliminar1} using as exponents, respectively $\frac{2}{3}$ and $\frac{1}{3}$, we are able to remove the $X_\rho$-factor:
\begin{equation}
\label{eq:P}
\mu_1\lessapprox C_1^{\frac{C}{\varepsilon_1}}q^{-\frac{7}{24}}\lambda_1^{-\frac{25}{24}}\delta^{-1}X_1^{\frac{1}{18}}\rho^{\frac{1}{6}}=:P.
\end{equation}
\subsubsection{The Guth--Zahl estimate $G$}
\label{Gestimate}
Next, we use the Guth--Zahl trilinear estimate presented in Theorem~\ref{trilinearKakeya}. For each relevant
$\delta$-cube $Q$, a $\gtrapprox1$ proportion of triples in
$\mathbb T_{Y_1}(Q)^3$ satisfies
$$
\left|
V(T')\wedge V(T'')\wedge V(T''')
\right|
\gtrsim\rho.
$$
Note that since we are applying directly to the refined family $(\mathbb{T}_1,Y_1)$, it is not expected any factor of $M$ to appear, differently from the derivation done for the $R$-estimate. Thus, the Guth--Zahl trilinear inequality implies
$$
|E_1|
\mu_1^{\frac{13}{9}}
\rho^{\frac{4}{9}}
\lessapprox
\delta^{-\frac{1}{3}}
X_1^{\frac{4}{3}}.
$$
Using \eqref{vera} and raising the whole estimate to the $\frac{9}{4}$-power gives
\begin{equation}
\mu_1
\lessapprox
\lambda_1^{-\frac{9}{4}}
\rho^{-1}
\delta^{-\frac{3}{4}}
X_1^{\frac{3}{4}}=: G.
\label{eq:G-final}
\end{equation}
Note that there is no additional dependence on the  two-ends constant $C_1$. This does not result in an obstruction to the interpolation, it only means that the constant $C$ is not affected by the weight given to $G$-estimate.
\subsubsection{Interpolation}
Now, notice that:
$$\mu_1=\mu_1^{\frac{2}{267}}\mu_1^{\frac{81}{89}}\mu_1^{\frac{22}{267}}\lessapprox G^{\frac{2}{267}}P^{\frac{81}{89}}R^{\frac{22}{267}}.$$
Calculating all the resulting exponents, we obtain
$$
\mu_1
\lessapprox C_1^{\frac{C}{\varepsilon_1}}
q^{-\frac{45}{356}}
\lambda_1^{-\frac{90}{89}}
\delta^{-\frac{87}{89}}
X_1^{\frac{41}{534}}.
$$
By construction $\lambda_1\lessapprox \lambda_{\rho}$, so we have
$q^{-\frac{45}{356}}
\lessapprox
1.$ We conclude that
\begin{equation}\label{eq:mu1-final}
\mu_1
\lessapprox C_1^{\frac{C}{\varepsilon_1}}
\lambda_1^{-\frac{90}{89}}
\delta^{-\frac{87}{89}}
X_1^{\frac{41}{534}}.
\end{equation}
Uniform fine multiplicity gives
\begin{equation}
\label{vera}
\mu_1|E_1|\approx \sum_{T\in\mathbb{T}_1}|Y_1(T)|\sim \lambda_1 X_1\implies |E_1|
\approx
\frac{\lambda_1X_1}{\mu_1},
\text{ with }
E_1
:=
\bigcup_{T\in\mathbb T_1}Y_1(T).
\end{equation}
Using \eqref{eq:mu1-final},
$$
|E_1|
\gtrapprox C_1^{-\frac{C}{\varepsilon_1}}
\lambda_1^{\frac{179}{89}}
\delta^{\frac{87}{89}}
X_1^{\frac{493}{534}}.
$$
The fine refinement retains a logarithmic/subpower fraction of the original shading:
\begin{equation}
\label{refinementdensity}
\lambda_1X_1\sim \sum_{T\in\mathbb{T}_1}|Y_1(T)|\gtrapprox \sum_{T\in \mathbb{T}}|Y(T)|
\sim
\lambda X,
\qquad
X:=\delta^3\#\mathbb T.
\end{equation}
Recall that it also satisfies
$\lambda_1\gtrapprox\lambda$. Consequently, by monotonicity:
$$
\left|
\bigcup_{T\in\mathbb T}Y(T)
\right|
\gtrapprox
C_1^{-\frac{C}{\varepsilon_1}}\lambda^{\frac{179}{89}}
\delta^{\frac{87}{89}}
(\delta^3\#\mathbb T)^{\frac{493}{534}},
$$
for some constant $C>0$. This proves that
$\mathbf{RT}
\left(
\frac{269}{89},
\frac{179}{89},
\frac{493}{534}
\right)$ holds. We are in the conditions of Proposition~\ref{RTtoTE} with $d=\frac{269}{89}$ and $b=\frac{493}{534}.$ Hence, we end up obtaining that the estimate
$\mathbf{TE}
\left(
\frac{269}{89},
\frac{179}{89},
\frac{493}{534}
\right)$ is valid.
\end{proof}
\begin{remark}
The step passing from $C_1=100$ to an arbitrary $C_1$ is crucial in the next subsection. We will ignore this issue for now and return to it in Section~\ref{bootstrapsection}, proving explicitly the missing argument in Lemma~\ref{lem:quantitative-TE}.
\end{remark}

\subsection{Restriction consequence of the improved two-ends estimate}
\label{subsec:restriction-consequence}

We briefly explain how the incidence estimate
$\mathbf{TE}
\left(
\frac{269}{89},
\frac{179}{89},
\frac{493}{534}
\right)$
improves the restriction exponent through the decoupling-incidence
argument of \cite{BCCLXZ}. 

The incidence estimate established above is formulated using the stronger Katz--Zahl two-ends condition. However, observe that it nevertheless supplies the incidence input required in the two-ends scenario of the decoupling-incidence argument of Borges et al. In their argument, after the two-ends reduction, the relevant shading is $(\eta^{2},\eta^{4}, c)$-two-ends in the sense of Wang--Wu, for an arbitrarily small parameter $\eta>0$ and some positive constant $c$ (see Section $3.5.$ of the same reference). A $\lambda$-dense $(\eta^{2},\eta^{4},c)$-two-ends shading satisfies the Katz--Zahl $(\eta^{4},C_{1})$-two-ends condition with  $C_{1}\lesssim\delta^{-\eta^{6}}$. First note that, due to the $\lambda$-denseness property, the assumed condition can be read as follows:
$$\#\{Q:Q\subset Y(T)\cap J\}\leq c\delta^{\eta^4}\frac{|Y(T)|}{\delta^4}\sim \delta^{\eta^4}\lambda\delta^{-1}.$$
Indeed, if $\delta\le r\le\delta^{\eta^{2}}$, a ball $B(x,r)$ intersects the tube inside $O(1)$ many $\delta\times \delta^{\eta^2}$-subtubes $J$. In this case, the weak two-ends condition gives
\[
\#\{Q:Q\subset Y(T)\cap B(x,r)\}
\lesssim
\delta^{\eta^{4}}\lambda\delta^{-1}
\le
r^{\eta^{4}}\lambda\delta^{-1}.
\]
If $\delta^{\eta^{2}}\le r\le1$, the trivial bound and
$r^{\eta^{4}}\ge\delta^{\eta^{6}}$
give
\[
\#\{Q:Q\subset Y(T)\cap B(x,r)\}
\leq \#\{Q:Q\subset Y(T)\}\sim \lambda\delta^{-1}\leq
\delta^{-\eta^{6}}r^{\eta^{4}}\lambda\delta^{-1}.
\]
We may therefore apply Theorem~\ref{thm:improved-two-ends} with two-ends parameters satisfying $\varepsilon_{1}=\eta^{4}$ and $C_{1}\lesssim\delta^{-\eta^{6}}$.
The corresponding dependence on the two-ends constant contributes only $C_{1}^{-\frac{C}{\eta^{4}}} \gtrsim \delta^{C\eta^{2}}$. Choosing $\eta>0$ sufficiently small relative to the prescribed $\varepsilon>0$, this loss is absorbed into the $\delta^{\varepsilon}$-loss. This gives the incidence estimate in the sense of Wang--Wu with the same parameters $d,a$ and $b$, and thus the decoupling methods of \cite{BCCLXZ} are allowed to be used. Note that, by setting $d=\frac{269}{89}$ and $b=\frac{493}{534}$, the incidence estimate corresponds to $\mathbf{TE}
\left(d
,\frac{d+1}{2},
b
\right).$ The remark done in Definition $1.2$ of the same reference already concludes our main goal. We guide the reader to their detailed explanation. A better illustration of this black-boxed argument will be done when studying the off-diagonal case in Section~\ref{offdiagonalsharpness}. After applying Tao's $\varepsilon$-removal argument in \cite{Tao1999BochnerRiesz}, we obtain that the localized restriction estimate obtained by the previous method implies Theorem~\ref{thm:main-restriction}.
\subsection{Sharpness of the interpolation technique}
\label{diagsharpness}
We now prove that one cannot expect to increase the range without exploring further the geometry of the construction. Note that the question on Appendix A of \cite{KZ} is not solved by this argument. That is, interpolation does not guarantee that with the same geometry other better estimates cannot be achieved.
\begin{theorem}
The interpolation technique over the estimates $G,P$ and $R$ cannot establish
$\mathbf{TE}\left(d,\frac{d+1}{2},b\right),$ for any
$d>\frac{269}{89}$. Consequently, through the incidence-decoupling implication, a restriction estimate \eqref{restriction} below $p>\frac{1250}{447}$ cannot be proved.
\end{theorem}
\begin{proof}
Consider an arbitrary geometric interpolation
$$
I=G^wP^uR^z,
\qquad
u,w,z\ge0,
\qquad
u+w+z=1.
$$
Before eliminating the coarse scale $\rho$, write the resulting estimate in the form
\begin{equation}
\mu_1
\lessapprox C_1^{\frac{C}{\varepsilon_1}}
q^\alpha
\lambda_1^{-y}
\delta^{-\kappa}
\rho^\beta
X_1^\gamma.
\label{eq:general-interpolated-form}
\end{equation}
A direct calculation gives
\begin{align*}
y
&=
\frac{9}{4}w
+
\frac{25}{24}u
+
\frac{9}{16} z=\frac{9}{16}+\frac{23}{48}u+\frac{27}{16}w,
\\
\alpha
&=
-\frac7{24}u+\frac{27}{16}z,
\\
\beta
&=
-w+\frac{u}{6}-\frac74z
=-\frac{7}{4}+\frac{23}{12}u+\frac{3}{4}w,
\\
\gamma
&=
\frac{3}{4}w+\frac{u}{18}+\frac{z}{4}=\frac{1}{4}-\frac{7}{36}u+\frac{w}{2},
\\
\kappa
&=
\frac{3}{4}w+u+\frac34z=\frac{3}{4}+\frac{u}{4}.
\end{align*}

The crucial observation is that the $\rho$-factor may be eliminated
uniformly using only the fact that $\delta\le\rho\le1$.
This leads to the following inequality of exponents.

\begin{lemma}[The $G,P,R$ supporting inequality]
\label{lem:PLH-supporting-inequality}
After eliminating the $\rho$-factor in
\eqref{eq:general-interpolated-form} using the condition
$\delta\le\rho\le1$, write the resulting bound as
\[
\mu_1
\lessapprox C_1^{\frac{C}{\varepsilon_1}}
q^\alpha
\lambda_1^{-y}
\delta^{-x}
X_1^\gamma.
\]
Then,
$46x+3y\ge48$.
Equality holds if and only if
$u=\frac{21-9w}{23}$,
which only happens if the interpolated $\rho$-power vanishes.
\end{lemma}

\begin{proof}
There are two possible cases. Suppose first that $\beta\ge0$, i.e. $u\ge\frac{21-9w}{23}$.
Since $\rho\le1$,
$\rho^\beta\le1$, the favorable $\rho$-factor can be discarded and one can take $x=\kappa$. Using the previous derivations,
$$
46x+3y-48=
46\kappa+3y-48=
\frac{207}{16}\left(u-\frac{21-9w}{23}\right)\ge0.
$$
Suppose instead that
$\beta<0$, i.e.
$u<\frac{21-9w}{23}$. Since $\rho\ge\delta$ and $\beta<0$,
$\rho^\beta\le\delta^\beta$.
Hence, $x=\kappa-\beta$.
Therefore,
$$
46x+3y-48=
46(\kappa-\beta)+3y-48=
-\frac{3611}{48}\left(u-\frac{21-9w}{23}\right)>0.
$$
Thus, in all cases,
$46x+3y\ge48$.
Moreover, equality in either case is possible precisely when
$u-\frac{21-9w}{23}=0$,
which is equivalent to
$\beta=0$.
\end{proof}
We get that the interpolation ultimately gives
$$
\mu_1
\lessapprox C_1^{\frac{C}{\varepsilon_1}}
\lambda_1^{-y-\max\{0,\alpha\}}
\delta^{-x}
X_1^\gamma,
$$
after discarding any favorable auxiliary factors (note that we used the fact that $\lambda_1q\lessapprox 1$ in the case $\alpha>0$). Using \eqref{vera}, one obtains
\begin{equation}
|E_1|
\gtrapprox C_1^{-\frac{C}{\varepsilon_1}}
\lambda_1^{1+y+\max\{0,\alpha\}}
\delta^x
X_1^{1-\gamma}.
\label{eq:final-general-volume}
\end{equation}
Set $d:=4-x$. In order for \eqref{eq:final-general-volume} to imply a
restriction-compatible estimate of the form $\mathbf{TE}
\left(
d,\frac{d+1}{2},1-\gamma
\right)$, we must have
$1+y+\max\{0,\alpha\}
\le
\frac{d+1}{2}$. Since $d=4-x$, this implies that $2y+x\le3$. Combining everything, we get by using Lemma~\ref{lem:PLH-supporting-inequality}:
\begin{equation}
\label{vera1}
48\le 46x+3y \le
46x+\frac{3}{2}\left(3-x\right)=
\frac{89x+9}{2}.
\end{equation}
Consequently,
$$
x\ge\frac{87}{89}\quad \implies \quad d=4-x
\le
\frac{269}{89}.
$$
We now show that the barrier is attained: equality in the first inequality of \eqref{vera1} requires $u=\frac{21-9w}{23}$. Since $u+w+z=1$, $z=\frac{2-14w}{23}$. Along this line of interpolation,
\begin{equation}
x=\frac{90-9w}{92}, \qquad
y=1+\frac{3}{2}w.
\end{equation}
The restriction-compatibility condition
$2y+x\le3$
therefore becomes $w\le\frac{2}{267}$. Since $x$ decreases with the increasing of $w$, maximizing the dimension $d=4-x$ is equivalent to
maximizing $w$. Hence, the optimal choice is
\begin{equation}
\label{vera4}
w=\frac{2}{267},
\qquad
u=\frac{81}{89},
\qquad
z=\frac{22}{267}.
\end{equation}
\end{proof}
\begin{remark} Since the optimal exponents in \eqref{vera4} were the ones used to get the restriction improvement in Theorem~\ref{thm:main-restriction}, one shall not expect a better exponent range without extracting new estimates using the same geometry derived by Katz--Zahl.
\end{remark}

\section{Further consequences: Bochner--Riesz conjecture}

In this section, we record the improvement for the Bochner--Riesz problem in $\mathbb R^4$. This will follow the same reasoning as Remark~$2.1$ of \cite{BCCLXZ}. By the classical reduction from the Bochner--Riesz problem to
Carleson--Sj\"olin oscillatory integral estimates, it suffices to establish appropriate $L^p$-bounds for operators of the form
\begin{equation}
\label{vera5}
    S_N f(x)
    :=
    \int_{\mathbb{R}^4}
    e^{2\pi i N |x-y|}
    a(x,y) f(y)\,dy,
\end{equation}
where $a \in C_c^\infty(\mathbb{R}^4\times\mathbb{R}^4)$ is supported away from the diagonal $\{x=y\}$. We briefly explain how the decoupling-incidence argument used to obtain Theorem~\ref{thm:main-restriction} also applies to \eqref{vera5}.

\subsection{Reduction to a Carleson--Sj\"olin operator.}
The singular behavior of the Bochner--Riesz multiplier is localized near the unit sphere. After a dyadic decomposition and the usual stationary-phase reduction, the relevant localized pieces are modeled by Carleson--Sj\"olin operators whose phase is the Euclidean distance function $\phi(x,y)=|x-y|$. Thus, by the standard reductions in \cite{CS72,GHI19}, the Bochner--Riesz estimate follows from the uniform estimates for $S_N$.

\subsection{Rescaling to the physical scale $N$.}
Define
\[
    S^N f(x):=S_N f\left(\frac{x}{N}\right).
\]
Then,
\[
    S^N f(x)
    =
    \int_{\mathbb{R}^4}
    e^{2\pi i |x-Ny|}
    a\left(\frac{x}{N},y\right) f(y)\,dy.
\]
This rescaling converts the original high-frequency problem at frequency $N$ into a large-scale wave-packet problem on the spatial scale $N$. At this scale, the relevant wave-packets are localized to tubes of length comparable to $N$ and transverse radius comparable to $N^{\frac{1}{2}}$.

\subsection{Refined decoupling.}
The phase $\phi(x,y)=|x-y|$
is a positive-definite H\"orma-\-nder phase. Refined decoupling estimates for positive-definite H\"ormander operators are also known to hold by the work developed in \cite{ILX22}.
Consequently, the refined decoupling estimate of Theorem~2.2 in \cite{BCCLXZ} applies
to the operator $S^N$. More precisely, after decomposing $f$ into wave-packets $f=\sum_T f_T$,
the refined decoupling inequality controls $S^Nf$ on a collection of
$N^{\frac{1}{2}}$-balls in terms of the individual wave-packet contributions $S^Nf_T$, with an additional gain depending on the maximal number of wave-packet tubes passing through each ball. Thus, the analytic problem
is reduced to controlling the multiplicity of the associated tube
family.

\subsection{Wave-packet geometry and incidence estimates.}
The work in \cite{GWX25} shows that the Wang--Wu decoupling-incidence method extends to positive-definite H\"orma-\-nder
operators by means of the standard $4\times 4$ wave-packet decomposition. For a general variable-coefficient H\"ormander phase, the corresponding
wave-packets tend to follow curved bicharacteristics. In the present case, however, the phase is the Euclidean distance function $|x-y|$. The associated bicharacteristics are Euclidean geodesics, and hence straight
lines. Therefore, the wave-packets of $S^N$ are concentrated on straight tubes, with exactly the same geometric structure as the tube family appearing in the restriction argument of this paper. In particular, the tube family associated to $S^N$ satisfies the same incidence estimate as in Theorem~\ref{thm:improved-two-ends}.

\subsection{Combining refined decoupling with incidence geometry.}
The refined decoupling estimate controls the oscillatory contribution in terms of the tube multiplicity, while Theorem~\ref{thm:improved-two-ends} controls that multiplicity geometrically. Hence, the same decoupling-incidence argument (after further reduction from Wang--Wu) used in the proof of Theorem~\ref{thm:main-restriction} applies to $S^N$. Carrying out the same induction and interpolation arguments therefore gives
\begin{equation}
    \|S^N f\|_{L^p(\mathbb{R}^4)}
    \leq
    C_\varepsilon N^\varepsilon
    \|f\|_{L^p(\mathbb{R}^4)},
\end{equation}
for every $p\geq 2+ \frac{356}{447}$
and every $\varepsilon>0$, where $C_\varepsilon$ is independent of
$N$.
\begin{remark}
This strategy is the higher-dimensional analog of the one developed in \cite{GWX25}, where the same techniques were used to improve the three-dimensional Bochner--Riesz exponent to $\frac{22}{7}$.
\end{remark}

We therefore obtain the following:
\begin{proposition}[Bochner--Riesz consequence]
\label{prop:BR}
Let $B_\delta$ denote the Bochner--Riesz operator in $\mathbb R^4$. Then, $B_\delta:L^p(\mathbb R^4)\to L^p(\mathbb R^4)$
is bounded whenever
$\max\{p,p'\}\geq
2+ \frac{356}{447}$ and
$
\delta>
4\left|
\frac12-\frac1p
\right|
-\frac12.$
\end{proposition}

\section{Proof of Corollary~\ref{offdiagresult} and Sharpness}
\subsection{The $b$-dependent restriction estimate}
\label{offdiagonalsharpness}
We start by proving the most important result of this section. This boils down to making explicit the observation done in Remark~$1.6$ of \cite{GWX25}.
\begin{proposition}
\label{b-dependentdec}
Suppose that the assumption $\mathbf{TE}\left(d,\frac{d+1}{2},b\right)$ holds for the tube families arising from the Wang--Wu reduction and $0\leq b\leq 1$. Then, with $p_d:=\frac{2(d+4)}{d+2},$ one has
\begin{equation}
\label{offdiagestimate}
\|E_Sf\|_{L^{p_d}(B(0,R))}
\lesssim_\varepsilon
R^\varepsilon
\|f\|_{L^2(B^3(0,1))}^{1-\frac{2b}{d+4}}
\left(
\sup_\theta\|f_\theta\|_{L^2_{\mathrm{avg}}}
\right)^{\frac{2b}{d+4}}.
\end{equation}
If, in addition, $b=\frac{d-1}{3}$ is satisfied, then the corresponding local $L^q\to L^p$ estimates hold throughout $p\ge p_d$ and $q'\le \frac35p.$
\end{proposition}
\begin{proof}
We follow the proof of the localized restriction estimate in Theorem~$3.1$ of \cite{BCCLXZ}, which combines the two-ends multiplicity estimate with refined decoupling. The additional bookkeeping is that, instead of discarding the normalized direction-count factor $R^{-\frac{3}{2}}\frac{\#\mathbb T}{m}$,
we maintain it and express it in terms of the $L^2$-concentration of $f$.\\ Some technical details are skipped, and we guide the reader for the checking of them in the previously mentioned proof.

Let $\mathbb T$ denote the retained tube family and set $
X:=R^{-\frac{3}{2}}\#\mathbb T.$ Let $\lambda_Y$ and $\mu_Y$ be, respectively, the average shading density and average multiplicity of the shaded tube family. Combining a similar expression as \eqref{vera} with $\mathbf{TE}(d,\frac{d+1}{2},b)$ gives
\begin{equation}
\label{mestimate}
\mu_Y
\lessapprox
m^b\lambda_Y^{-\frac{d-1}{2}}
\delta^{d-4}
(\delta^3\#\mathbb T)^{1-b}.
\end{equation}
Until now, we have only dealt with $1$-parallel families of tubes. For decoupling reasons, it is of our interest to instead study $m$-parallel families. Consider a $m$-parallel family of tubes $(\mathbb{T}_m,Y_m)$ with its associated shading. By taking one tube in each neighborhood of possible directions, we obtain a $1$-parallel family and its shading $(\mathbb{T}_1, Y_1)$ satisfying $\#\mathbb{T}_1\gtrsim \frac{\#\mathbb{T}_m}{m}$. Then, an estimate like \eqref{mestimate} follows from applying the incidence estimate to $\mathbb{T}_1$, using monotonicity over the inclusion $E_{Y_1}\subset E_{Y_m}$ and converting the resulting expression into an upper-bound of the average multiplicity $\mu_{Y_m}$.

After the usual wave-packet decomposition and dyadic pigeonholing, we may assume that all active packets have comparable $L^2$-norm and that every active $R^{-\frac{1}{2}}$-cap contains
$\sim m$ wave-packets. Moreover, it suffices to estimate $E_Sf$ on a union $E_Y$ of relevant cubes. 
By Proposition~$3.1$ of \cite{BCCLXZ} and Markov's inequality, after discarding a subset of $E_Y$ of relative size $O(R^{-\varepsilon})$, we may restrict to a set $E_\mu\subset E_Y$ in which $\#\mathbb{T}_Y(Q) \lessapprox \mu$ for all $\delta$-cubes such that $Q\subset E_\mu$, and
$$
\mu
\lesssim_\varepsilon
R^\varepsilon
m^b\lambda_Y^{-\frac{d-1}{2}}
\delta^{d-4}
(\delta^3\#\mathbb T)^{1-b}.
$$
By substituting $\delta=R^{-\frac{1}{2}}$, we obtain
$$
\frac{\mu}{m}
\lesssim_\varepsilon
R^\varepsilon
\lambda_Y^{-\frac{d-1}{2}}
R^{\frac{4-d}{2}}
\left(\frac{X}{m}\right)^{1-b}.
$$
This is the point at which the diagonal and non-diagonal arguments diverge. Indeed, if $K$ denotes the number of active $R^{-\frac{1}{2}}$-caps, then $\#\mathbb T\sim Km.$ There are only $O(R^{\frac{3}{2}})$ such caps on the three-dimensional frequency surface, so $K\lesssim R^{\frac{3}{2}}.$ Consequently,
$$
\frac{X}{m}
=
R^{-\frac{3}{2}}\frac{\#\mathbb T}{m}
\sim
R^{-\frac{3}{2}}K
\lesssim1.
$$
Thus, the usual diagonal argument can simply discard $\left(\frac{X}{m}\right)^{1-b}.$ This is what implicitly happens from Equation $(47)$ to $(48)$ in \cite{BCCLXZ}. For the off-diagonal estimate we instead retain it. Set $A(f):=
\sup_\theta\|f_\theta\|_{L^2_{\mathrm{avg}}},$ where we set $f_\theta:=f \eta_\theta$ for a cap $\theta$ arising from the wave-packet decomposition done over $B^3(0,1)$, $\{\eta_\theta\}_\theta$ a smooth partition of unity subordinated to the cap covering and
$$\| f_\theta\|^2_{L^2_{\text{avg}}}:= \frac{1}{|\theta|}\|f_\theta\|_{L^2(B^3(0,1))}^2.$$
Let $\varrho^2$ denote the common squared $L^2$-mass of the retained packets. By almost orthogonality, $
\|f\|_{L^2(B^3(0,1))}^2\sim Km\varrho^2.$ Since a $R^{-\frac{1}{2}}$-cap has measure $\sim R^{-\frac{3}{2}}$, $
A(f)^2
\sim
R^{\frac{3}{2}}m\varrho^2.$ Combining everything gives
$$
\frac{\mu}{m}
\lesssim_\varepsilon
R^\varepsilon
\lambda_Y^{-\frac{d-1}{2}}
R^{\frac{4-d}{2}}
\left(\frac{\|f\|_{L^2(B^3(0,1))}}{A(f)}\right)^{2(1-b)}.
$$
Next, we use the following two results described in \cite{BCCLXZ}. Firstly, the estimate presented in Equation $(47)$ obtained by a localized refined-decoupling argument that combined with our assumptions leads to the following:
\begin{equation}
\label{refinedestimate}
\left\|\sum_{T\in\mathbb{T}} E_S(f_T)\mathds{1}_{Y(T)}\right\|_{L^{\frac{10}{3}}(E_\mu)}
\lesssim_\varepsilon
R^{-\frac{3}{10}+\varepsilon}
\left(\frac{\mu}{m}\right)^{\frac{1}{5}}
\|f\|_{L^2(B^3(0,1))}^{\frac{3}{5}}A(f)^{\frac{2}{5}}.
\end{equation}
Secondly, the $L^2$-orthogonality bound stated in Lemma~$2.3$:
\begin{equation}
\label{L2orthogonality}
\left\|\sum_{T\in \mathbb{T}} E_S(f_T)\mathds{1}_{Y(T)}\right\|_{L^2(E_\mu)}
\lesssim
(\lambda_Y R)^{\frac{1}{2}}\|f\|_{L^2(B^3(0,1))}.
\end{equation}
Let $0\le\tau\le1$ be determined by $\frac1p=\frac12-\frac{\tau}{5}$, for $2\leq p\leq \frac{10}{3}$. By Riesz--Thorin theorem, interpolating \eqref{refinedestimate} and \eqref{L2orthogonality} yields
$$
\|E_Sf\|_{L^p(E_\mu)}
\lesssim_\varepsilon
R^{\varepsilon+\frac12-\frac{4\tau}{5}}
\lambda_Y^{\frac{1-\tau}{2}}
\left(\frac{\mu}{m}\right)^{\frac{\tau}{5}}
\|f\|_{L^2(B^3(0,1))}^{1-\frac{2\tau}{5}}
A(f)^{\frac{2\tau}{5}}.
$$
Choose $p=p_d$. A direct calculation gives 
$$
\|E_Sf\|_{L^{p_d}(E_\mu)}
\lesssim_\varepsilon
R^\varepsilon
\|f\|_{L^2(B^3(0,1))}^{1-\frac{2b}{d+4}}
A(f)^{\frac{2b}{d+4}}.
$$
This proves \eqref{offdiagestimate}. To pass from the mixed estimate to the stronger $L^q\to L^p$ bound, we reproduce explicitly the standard level-set decomposition mentioned in Remark~$1.2$ of \cite{GWX25}. Let $\alpha=\frac{2b}{d+4}$ and define $q_d$ by $\frac{1}{q_d}=\frac{1-\alpha}{2}$. By homogeneity, assume $\|f\|_{L^{q_d}(B^3(0,1))}=1$. Decompose
$$
f=\sum_{k\in \mathbb{Z}} f_k,
\qquad
f_k=f\,\mathds{1}_{\{2^k\le |f|<2^{k+1}\}}.
$$
If $E_k=\operatorname{supp}(f_k)$, then $\|f_k\|_{L^2(B^3(0,1))}\lesssim 2^k|E_k|^{\frac{1}{2}}$ and $A(f_k)\lesssim 2^k.$
Hence, the mixed estimate \eqref{offdiagestimate} gives
$$
\|E_Sf_k\|_{L^{p_d}(B(0,R))}
\lesssim_\varepsilon
R^\varepsilon
2^k|E_k|^{\frac{1-\alpha}{2}}
=
R^\varepsilon
2^k|E_k|^{\frac{1}{q_d}}.
$$
Since $\|f\|_{L^{q_d}(B^3(0,1))}=1$, $2^{k{q_d}}|E_k|\lesssim1,$ and therefore
\begin{equation}
\label{dyadicestimate}
\|E_Sf_k\|_{L^{p_d}(B(0,R))}
\lesssim_\varepsilon R^\varepsilon.
\end{equation}
Now, we sum over all the possible dyadic levels. By the trivial bound
$$
\|E_Sg\|_{L^{p_d}(B(0,R))}
\lesssim R^{\frac{4}{p_d}}\|g\|_{L^1(B^3(0,1))}, \,\, \text{ for any function } g\in L^1(B^3(0,1)),
$$
the contributions where $|f|<R^{-C}$ or $|f|>R^C$ are negligible for $C$ sufficiently large. Thus, only $O(\log(R))$ dyadic levels remain. Applying \eqref{dyadicestimate} with $\frac{\varepsilon}{2}$-exponent,
$$
\|E_Sf\|_{L^{p_d}(B(0,R))}
\lesssim_\varepsilon
(\log(R))R^{\frac{\varepsilon}{2}}
\lesssim_\varepsilon
R^\varepsilon.
$$
Restoring the normalization gives
$$
\|E_Sf\|_{L^{p_d}(B(0,R))}
\lesssim_\varepsilon
R^\varepsilon\|f\|_{L^{q_d}(B^3(0,1))}.
$$
It remains to pass from the endpoint estimate to the full source range. Assuming now that $b=\frac{d-1}{3}$, we obtain
$$
q_d'
=
\frac{2(d+4)}
{d+4+\frac{2(d-1)}3}
=
\frac{6(d+4)}{5(d+2)}=\frac35p_d.$$
Thus, the endpoint obtained above lies on the conjectural scaling line proposed by Stein in \cite{Ste79}. The Stein--Tomas estimate $E_S:L^2\to L^{\frac{10}{3}}$ is also part of the same line. Interpolating the endpoint estimate with Stein--Tomas therefore gives
$$
q'=\frac35p,
\qquad
p_d\le p\le\frac{10}{3}.
$$
For $p\ge\frac{10}{3}$, the interpolation between Stein--Tomas and the trivial estimate $E_S:L^1\to L^\infty$ continues to be on the same scaling line. Hence, we obtain that our estimate is valid for
$$
q'=\frac35p,
\qquad
\text{ for every }p\ge p_d.
$$
Finally, since the frequency domain has finite measure, Hölder's inequality gives all source exponents inside the scaling line.
\end{proof}
Now, we shall remember the second part of Lemma~\ref{hairbrushest}. This will provide in the next subsection a stronger estimate while in the intermediate robust transverse step before applying Proposition~\ref{RTtoTE}.
\subsubsection{The Hairbrush estimate \texorpdfstring{$H$}{H}}
We wish to apply the second part of Lemma~\ref{hairbrushest} to the final fine family $(\mathbb T_1,Y_1)$. First, since $Y_1(T)\subseteq Y(T)$ for
every $T\in\mathbb T_1$, the original $(\varepsilon_1,C_1)$-two-ends condition implies that
\[
\#\{Q:Q\subset Y_1(T)\cap B(x,r)\}
\le C_1r^{\varepsilon_1}
\lambda_Y\delta^{-1},
\]
for every $\delta\le r\le1$.
Therefore, $(\mathbb T_1,Y_1)$ satisfies the $(\varepsilon_1,C_{1}')$ two-ends condition, with $C_{1}'
=\max\left\{1,
C_1\frac{\lambda_Y}{\lambda_1}\right\}$. By property $(vi)$ of Lemma~\ref{KZconstruction}, we have
$\lambda_1\gtrapprox\lambda$. Consequently, for every $\sigma>0$,
$C_{1}'
\lesssim_\sigma C_1\delta^{-\sigma}$. Moreover, by
Corollary~\ref{cor:fine-transverse},
the fine family satisfies the angular non-concentration hypothesis required. Let $C_0>0$ be the absolute constant in the exponent of the two-ends constant
in Lemma~\ref{hairbrushest}. Applying the lemma with scale-loss parameter $\eta>0$ gives
\[
\mu_1\lesssim_{\eta,\varepsilon_1}
(C_{1}')^{\frac{C_0}{\varepsilon_1}}\delta^{-\eta}
\lambda_1^{-\frac{3}{4}}
\delta^{-1}X_1^{\frac{1}{2}}\lesssim_{\eta,\varepsilon_1,\sigma}
C_1^{\frac{C_0}{\varepsilon_1}}
\delta^{-\eta-\frac{C_0\sigma}{\varepsilon_1}}
\lambda_1^{-\frac{3}{4}}
\delta^{-1}X_1^{\frac{1}{2}}.
\]
For any prescribed $\varepsilon>0$,
we may choose $\eta,\sigma>0$
sufficiently small so that
$\eta+\frac{C_0\sigma}{\varepsilon_1}
<\varepsilon$, while also satisfying the parameter requirements of Corollary~\ref{cor:fine-transverse}. It follows that
\begin{equation}\label{eq:H}
\mu_1
\lessapprox
C_1^{\frac{C}{\varepsilon_1}}
\lambda_1^{-\frac{3}{4}}
\delta^{-1}X_1^{\frac{1}{2}}:=H.
\end{equation}

\subsubsection{Interpolation}
\label{off-diagonal}
Now, it is left to find the best interpolation between the estimates $G,P,H$ in order to achieve a reasonable two-ends estimate that improves the result so far known. We will first present the interpolation exponents and show that it improves the result in Corollary~$1.2$ of \cite{GWX25} for $n=4$. After that, we prove its sharpness. 

Interpolate the $G,P,H$-estimates with exponents $\frac{1}{14},$ $\frac{3}{7}$ and $\frac{1}{2}$, respectively. By a direct computation, we get:
$$
\mu_1 \lessapprox C_1^{\frac{C}{\varepsilon_1}} q^{-\frac{1}{8}}
\lambda_1^{-\frac{55}{56}}
\delta^{-\frac{55}{56}}
X_1^{\frac{55}{168}}.
$$
The factor $q=\frac{\lambda_{\rho}}{\lambda_3}$ satisfies $q\gtrapprox 1$, so it may be discarded. Therefore, we obtain
$$
\mu_1
\lessapprox
C_1^{\frac{C}{\varepsilon_1}} \lambda_1^{-\frac{55}{56}}
\delta^{-\frac{55}{56}}
X_1^{\frac{55}{168}}.
$$
Using a similar procedure as in \eqref{vera}, we get that
$\mathbf{RT}
\left(
\frac{169}{56},
\frac{111}{56},
\frac{113}{168}
\right)$ holds.
Since $\frac{113}{168}=\frac{\frac{169}{56}-1}{3}$, we have that $\mathbf{RT}
\left(
\frac{169}{56},
\frac{111}{56},
\frac{113}{168}
\right)$ implies
$\mathbf{TE}
\left(
\frac{169}{56},
\frac{111}{56},
\frac{113}{168}
\right)$. Since $0<\lambda_1\leq 1$ and $\frac{111}{56}<\frac{225}{112}$, we have that $\mathbf{TE}\left(
\frac{169}{56},
\frac{111}{56},
\frac{113}{168}
\right)$ implies $\mathbf{TE}\left(
\frac{169}{56},
\frac{225}{112},
\frac{113}{168}
\right)$. The latter is particularly interesting since it is of the form $\mathbf{TE}\left(d,\frac{d+1}{2},\frac{d-1}{3}\right)$. By applying Proposition~\ref{b-dependentdec}, we immediately obtain  Corollary~\ref{offdiagresult}. Note that $\frac{786}{281}\approx2.79715<2.8 $ so we get a strictly better range of values of $p$.
\subsection{Sharpness of the interpolation technique}
\label{offsharpness}
This section will follow a similar procedure to the one presented in Section~\ref{diagsharpness}. For this reason, we will only briefly discuss how the argument should go.
Consider an arbitrary geometric interpolation
\[
I=G^wP^uH^z,
\qquad
u,w,z\ge0,
\qquad
u+w+z=1.
\]
Its geometric mean gives a multiplicity estimate containing the factor $\rho^{\frac{u-6w}{6}}$. Using $\delta\le\rho\le1$, this factor can always be eliminated, possibly at the expense of an additional power of $\delta$. By a similar result as in Lemma~\ref{lem:PLH-supporting-inequality} substituting the $R$-estimate by the $H$ one, the optimal interpolation occurs precisely at the boundary $u=6w$, where the interpolated $\rho$-power vanishes and no additional $\delta$-loss is gained. For any choice of exponents, the power of $q=\frac{\lambda_{\rho}}{\lambda_3}$ is favorable. After applying uniform multiplicity, the resulting two-ends exponents are
\[
d=3+\frac{w}{4},
\qquad
b=\frac{1}{2}-\frac{w}{4}+\frac{4}{9}u.
\] 
We remark that the $X$-factor always satisfies $X=\delta^3\#\mathbb T\lesssim1.$ It follows that a lower-bound involving $X^b$ is at least as strong as the one with the critical exponent, up to an absolute constant, if we require $b\leq \frac{d-1}{3}$. Thus, 
$$ \frac{1}{2}-\frac{w}{4}+\frac{4}{9}u \le \frac{2+\frac{w}{4}}{3}\quad \implies \quad 8u-6w\leq3. $$
Combining both conditions gives $w\leq \frac{1}{14}$. But the value of $d$ increases with the parameter $w$. Therefore, the largest possible $d$ available from the three estimates is attained at $ w=\frac{1}{14}$. Equality forces $ u=6w=\frac{3}{7}$ and $z=1-u-w=\frac12$. Note that these exponents exactly match the ones we chose in each estimate, which ends the sharpness argument.
\section{Proof of Theorem~\ref{vera3}, Corollary~\ref{Kakeyalowerbound} and Sharpness}
\label{bootstrapsection}
\subsection{A general X-ray bootstrap}
\label{subsec:general-xray-bootstrap}

We now combine the fine/coarse refinement with the \L aba--Tao
X-ray estimate and the Guth--Zahl trilinear estimate. The heuristic on why this new estimate will turn out to be useful was already noted in Remarks $3.1$ and $4.1$ of \cite{BCCLXZ}. The bootstrap
mechanism is the same as the one appearing in the proof of
Lemma~38 in \cite{KZ}, but we retain the three parameters
$(a,\ell,x)$ since the two-ends estimates used below do not lie in the special one-parameter $\alpha$ family considered there. We first record the form of the \L aba--Tao X-ray estimate that will be
used later.

\subsubsection{The \L aba--Tao estimate $L$}
\label{Lestimate}

The discretized adjoint X-ray estimate in dimension four is given by restricting Theorem $1.2$ of \cite{LabaTao2001} to $n=4$, whose expression reads as follows:
\begin{equation}
\left\|
\sum_{T\in\mathcal A}\mathds{1}_T
\right\|_{L^{\frac{3}{2}}(\mathbb{R}^4)}
\lessapprox
\tau^{-\frac{1}{3}}
m^{\frac{5}{36}}
\left(
\tau^3\#\mathcal A
\right)^{\frac{7}{9}},
\label{eq:LT-R4-final}
\end{equation}
for a family $\mathcal{A}$ of essentially distinct $\tau$-tubes with at most $m$ tubes in each $\tau$-separated direction. Apply \eqref{eq:LT-R4-final} to $\mathcal{A}=\mathbb T_\rho$. The fine/coarse refinement gives essentially distinct coarse tubes, with $\approx A$ tubes in each relevant $\rho$-separated direction. Thus, $
m\approx A$ and $\tau=\rho$. We find that the following estimate holds:
\begin{equation}
\left\|
\sum_{T_\rho\in\mathbb T_\rho}
\mathds{1}_{T_\rho}
\right\|_{L^{\frac{3}{2}}(\mathbb{R}^4)}
\lessapprox
\rho^{-\frac{1}{3}}
A^{\frac{5}{36}}
{X_\rho}^{\frac{7}{9}}.
\label{eq:LT-coarse-final}
\end{equation}
Double counting gives 
\begin{equation}
\mu_\rho|E_\rho|\approx \sum_{T_\rho\in\mathbb{T_\rho}}|Y_\rho(T_\rho)|\sim \lambda_\rho X_\rho\implies |E_\rho|
\approx
\frac{\lambda_\rho X_\rho}{\mu_\rho},
\text{ with }
E_\rho
:=
\bigcup_{T_\rho\in\mathbb T_\rho}Y_\rho(T_\rho).
\end{equation}
Note that, on $E_\rho$, we have
$\sum_{T_\rho\in\mathbb T_\rho}\mathds{1}_{T_\rho}
\gtrapprox \mu_{\rho}.$ Take $x\in E_\rho$ and let $Q_\rho$ be a $\rho$-cube containing $x$. The pointwise bound follows after seeing that every tube counted in $\mathbb T_\rho(Q_\rho)$ contains $Q_\rho$ and hence contains $x$. For this reason, one obtains:
$$
\sum_{T_\rho\in\mathbb T_\rho}\mathds{1}_{T_\rho}(x)
\ge
\#\mathbb T_\rho(Q_\rho)
\approx
\mu_{\rho}.$$
This last relation implies
$$
\left\|
\sum_{T_\rho\in \mathbb{T}_{\rho}}\mathds{1}_{T_\rho}
\right\|_{L^{\frac{3}{2}}(\mathbb{R}^4)}
\gtrapprox
\mu_\rho|E_\rho|^{\frac{2}{3}} \approx
\mu_\rho^{\frac{1}{3}}
\lambda_{\rho}^{\frac{2}{3}}
X_{\rho}^{\frac{2}{3}}.
$$
Combining this with \eqref{eq:LT-coarse-final} yields
\begin{equation}
\mu_{\rho}
\lessapprox
\lambda_{\rho}^{-2}
\rho^{-1}
A^{\frac{5}{12}}
X_{\rho}^{\frac{1}{3}}=:L.
\label{eq:LT-coarse-multiplicity-final}
\end{equation}
Likewise the $G$-estimate, the \L aba--Tao estimate does not depend on $C_3$ (and subsequently there is no dependence to be tracked over the constant $C_1$).

We now formulate the parameterized version of the Katz--Zahl X-ray
bootstrap.
\begin{lemma}[General X-ray bootstrap]
\label{lem:general-xray-bootstrap}
Suppose that $
\mathbf{TE}(4-a,1+\ell,1-x)
$ holds, where
$
\ell\geq0$, 
$0\leq x\leq\frac7{16},$ and $a-1+\frac{4x}{7}\geq0.$ Set
$w
:= \left(a+\frac{4x}{7}\right)^{-1}$,
and define
\begin{equation}
a_*
:=
\frac34
+
\left(a-\frac34\right)w,
\,\,
\ell_*
:=
\left(
\ell+\frac{12x}{7}
\right)w
+
\frac94(1-w), \,\,
x_*:=
\frac{4x}{21}w
+
\frac34(1-w).
\label{bootstrapeq}
\end{equation}
Then,
$
\mathbf{RT}(4-a_*,1+\ell_*,1-x_*)
$ holds. If, in addition, $
x_*\leq\frac{a_*}3
$ is satisfied, then
$
\mathbf{TE}(4-a_*,1+\ell_*,1-x_*)
$ holds.
\end{lemma}

\begin{proof}
The proof follows the argument of Katz--Zahl for Lemma~38, retaining the
parameters $(a,\ell,x)$. At scale $\rho$, the planebrush estimate derived in Proposition $34$ of \cite{KZ} gives
\begin{equation}
\mu_{\rho}
\lessapprox
C_1^{\frac{C}{\varepsilon_1}}\lambda_\rho^{-\frac{1}{3}} X_\rho^{\frac{1}{3}} \rho^{-\frac{2}{3}}\left(\rho^3 \#\Omega\right)^{-\frac{1}{3}} D^{\frac{4}{3}}\sim 
C_1^{\frac{C}{\varepsilon_1}}\lambda_{\rho}^{-\frac{1}{3}}
\rho^{-\frac{2}{3}}
A^{\frac{1}{3}}
D^{\frac{4}{3}},
\label{eq:PB-bootstrap-final}
\end{equation}
for some constant $C>0$. Interpolating \eqref{eq:PB-bootstrap-final} with
\eqref{eq:LT-coarse-multiplicity-final}, with the respective exponents $\frac{3}{4}$ and $\frac{1}{4}$, gives
\begin{equation}
\mu_{\rho}
\lessapprox C_1^{\frac{C}{\varepsilon_1}}
\lambda_{\rho}^{-\frac{3}{4}}
\rho^{-\frac{3}{4}}
A^{\frac{17}{48}}
X_{\rho}^{\frac{1}{12}}
D.
\label{eq:PB-LT-final}
\end{equation}
Property $(xi)$ of Lemma~\ref{KZconstruction} together with
$\lambda_{\rho}\gtrapprox\lambda_1$ make \eqref{eq:PB-LT-final} imply
\begin{equation}
\mu_1
\lessapprox C_1^{\frac{C}{\varepsilon_1}}
\lambda_1^{-\frac{3}{4}}
\rho^{-\frac{3}{4}}
A^{\frac{17}{48}}
X_{\rho}^{\frac{1}{12}}
\widetilde\mu.
\label{eq:fine-PB-LT-final}
\end{equation}

We now wish to apply the assumed two-ends estimate $\mathbf{TE}(4-a,1+\ell,1-x)$ to the rescaled fine descendants, as in \eqref{eq:mu-tilde}. However, by Definition~\ref{TEdef},
the assertion $\mathbf{TE}(d,e,b)$ is formulated for families whose two-ends constant is bounded by $100$. Since our fine family may inherit a two-ends constant depending on $C_1$, we first record the following quantitative extension of the assumed estimate.

\begin{lemma}[Quantitative two-ends reduction]
\label{lem:quantitative-TE}
Suppose that $\mathbf{TE}(d,e,b)$ holds, where $d\le4$, $e\ge1$ and $0\le b\le1$. Let $(\mathbb T,Y)$ be a family of $s$-tubes in $\mathbb{R}^4$ pointing in $s$-separated directions, and suppose
that it satisfies the $(\varepsilon_1,C_1)$-two-ends condition, with $C_1\ge1$. Then, for every $\eta>0$,
\[
|E_Y|
\gtrsim_{\eta,\varepsilon_1,d,e,b}
C_1^{-\frac{8}{\varepsilon_1}}
s^{4-d+\eta}\lambda_Y^e
(s^3\#\mathbb T)^b.
\]
In particular, the exponent of $C_1$
is independent of $\eta$.
\end{lemma}

\begin{proof}
We follow the standard two-ends
localization argument from the proof of Proposition $11$ in \cite{KZ}, keeping track of the dependence on $C_1$. After the usual density uniformization and localization on balls containing at least a shaded-cube fraction of $r^{\frac{\varepsilon_1}{2}}$, we obtain a common scale $r$ satisfying $r\gtrsim_{\varepsilon_1}
C_1^{-\frac{2}{\varepsilon_1}}$, such that the localized shadings satisfy a fixed-constant two-ends condition after rescaling by $r^{-1}$. More precisely, the rescaled tubes have width $\frac{s}{r}$ (which must be bounded by $1$) and shading density bounded below by $\widetilde{\lambda}
\gtrsim\lambda_Y r^{\frac{\varepsilon_1}{2}-1}$, up to logarithmic losses.  Since the
original directions are $s$-separated, we may retain a subfamily with $\frac{s}{r}$-separated directions, losing at most a factor $r^3$ in the tube
count.

Applying $\mathbf{TE}(d,e,b)$ to
these rescaled families, and then
returning to the original scale,
introduces the additional factor
$r^{d-e\left(1-\frac{\varepsilon_1}{2}\right)-\frac{\eta}{2}}$. Since $d\le4$, $e\ge1$, and $0<r\le1$, we have, for $0<\eta<1$, $r^{d-e\left(1-\frac{\varepsilon_1}{2}\right)-\frac{\eta}{2}} \ge r^4
\gtrsim_{\varepsilon_1}
C_1^{-\frac{8}{\varepsilon_1}}.$ Summing over the disjoint localized
regions and using $0\le b\le1$,
we recover the required lower-bound.
The density uniformization and
pigeonholing losses are logarithmic.

Finally, if
$C_1\gtrsim_{\varepsilon_1}
s^{-\frac{\varepsilon_1}{2}}$,
the result follows directly from
$|E_Y|\gtrsim\lambda_Ys^3$, since $s^3\#\mathbb T\lesssim1$
and $C_1^{-\frac{8}{\varepsilon_1}}
s^{4-d+\eta}\lambda_Y^e
(s^3\#\mathbb T)^b
\lesssim_{\varepsilon_1}
\lambda_Ys^{8-d+\eta}
\leq\lambda_Ys^3$.
\end{proof}
Using a similar procedure as the one done in \eqref{anisotropic}, we can observe that the rescaled descendants satisfy
the $(\varepsilon_1, \widetilde{C_1})$
two-ends condition, where the constant satisfies
$\widetilde{C_1}
\lessapprox\max\left\{
1,C_1\frac{\lambda_Y}
{\widetilde{\lambda}}
\right\}$. Consequently, for every $\sigma>0$, $\widetilde{C_1}
\lesssim_\sigma C_1\delta^{-\sigma}$. We now apply
Lemma~\ref{lem:quantitative-TE}
to the rescaled fine descendants. Taking $d=4-a, e=1+\ell, b=1-x$ and recalling \eqref{descendantsrelation}, we obtain
\begin{equation}
\label{eq:rescaled-TE-final}
\widetilde\mu
\lesssim_{\eta,\varepsilon_1}
\widetilde{C_1}^{\frac{8}{\varepsilon_1}}
\left(\frac{\delta}{\rho}\right)^{-a-\eta}
\lambda_1^{-\ell}
\left(\frac{X_1}{X_\rho}\right)^x\lessapprox C_1^{\frac{8}{\varepsilon_1}}
\lambda_1^{-\ell}
\left(\frac{\delta}{\rho}\right)^{-a}
\left(\frac{X_1}{X_\rho}\right)^x.
\end{equation}
Substituting \eqref{eq:rescaled-TE-final} into
\eqref{eq:fine-PB-LT-final} yields
\begin{equation}
\mu_1
\lessapprox C_1^{\frac{C}{\varepsilon_1}}
\lambda_1^{-\ell-\frac{3}{4}}
\delta^{-a}
\rho^{a-\frac{3}{4}}
A^{\frac{17}{48}}
X_1^x
X_{\rho}^{\frac{1}{12}-x}.
\label{eq:bootstrap-bound-I-final}
\end{equation}
We also have the independent refinement estimate
$\mu_1\leq \rho^{-1}\widetilde\mu$ (see property $(ix)$ of Lemma~\ref{KZconstruction}). Similarly \eqref{eq:rescaled-TE-final} upgrades this estimate to the following:
\begin{equation}
\mu_1
\lessapprox C_1^{\frac{C}{\varepsilon_1}}
\lambda_1^{-\ell}
\delta^{-a}
\rho^{a-1}
X_1^x
X_{\rho}^{-x}.
\label{eq:bootstrap-bound-II-final}
\end{equation} Interpolate
\eqref{eq:bootstrap-bound-I-final} and
\eqref{eq:bootstrap-bound-II-final}
with exponents $
t=\frac{16x}{7}$ and 
$1-t,$ respectively. The hypothesis
$
0\leq x\leq\frac7{16}
$ ensures that $0\leq t\leq1$. Using \eqref{Arelation}, we conclude that
\begin{equation}
\mu_1
\lessapprox C_1^{\frac{C}{\varepsilon_1}}
\lambda_1^{-\ell-\frac{12x}{7}}
\delta^{-a}
\rho^{a-1+\frac{4x}{7}}
X_1^{\frac{4x}{21}}.
\label{eq:pre-G-final}
\end{equation}
Set $
r
:=
a-1+\frac{4x}{7}.
$ By hypothesis, $r\geq0$. Interpolate \eqref{eq:pre-G-final} and \eqref{eq:G-final} with exponents $w$ and $1-w$, respectively, where $w=(1+r)^{-1}$. When $w=1$, this is understood simply as using \eqref{eq:pre-G-final}. The $\rho$-power cancels and we obtain
\begin{equation}
\mu_1
\lessapprox C_1^{\frac{C}{\varepsilon_1}}
\lambda_1^{-\ell_*}
\delta^{-a_*}
X_1^{x_*},
\label{eq:bootstrap-output-multiplicity-final}
\end{equation}
where $a_*,\ell_*,x_*$ are given by the expressions in 
\eqref{bootstrapeq}. We now transfer the estimate from the refined family to the original
family. Since $\ell\geq0$ and $0<w\leq 1$, we have $
\ell_*\geq0.
$ Moreover, $
x_*\geq0.
$ Therefore, applying a similar procedure to the one in \eqref{refinementdensity}, one gets:
$$
|E|
\geq
|E_1|\gtrapprox
C_1^{-\frac{C}{\varepsilon_1}}\lambda_1X_1
\lambda_1^{\ell_*}
\delta^{a_*}
X_1^{-x_*}\gtrapprox
C_1^{-\frac{C}{\varepsilon_1}}\lambda^{1+\ell_*}
\delta^{a_*}
X^{1-x_*},
$$
for $E=\bigcup_{T\in\mathbb{T}}Y(T)$ and $\mathbb{T}_1$ to be the fine refinement of the family $\mathbb{T}$. Thus, we find that $\mathbf{RT}(4-a_*,1+\ell_*,1-x_*)$ holds. Finally, one can apply Proposition~\ref{RTtoTE} whenever
$
1-x_*
\geq
\frac{(4-a_*)-1}{3}$, which is equivalent to $x_*\leq\frac{a_*}3.$ Under this additional condition,
$\mathbf{TE}(4-a_*,1+\ell_*,1-x_*)$ holds.
\end{proof}

\begin{remark}
\label{specialization}
If one chooses $a=\alpha$,
$\ell=\frac{47}{28},$ and 
$x=\frac{\alpha}{3}$, then
$w=\frac{21}{25\alpha}$, and the preceding argument specializes to the interpolation mechanism in the first part of Lemma~$38$ proved by Katz--Zahl, after using the fact that $X^{1-x_*}\gtrsim X$. The purpose of Lemma~\ref{lem:general-xray-bootstrap} is only to retain
the general three-parameter dependence in order to obtain the best interpolation as possible between the estimates.
\end{remark}
\subsection{A second and third improved two-ends incidence estimates} First, we present the choices of interpolation and the calculations behind each iteration of the general bootstrap. After that, we prove that the proposed configuration is the sharpest one. Sharpness should be intended with the same reasoning as the one exposed in Section~\ref{diagsharpness}.  One can note that the bootstrap applied to our sharp interpolation cannot beat the Hausdorff dimension lower-bound for Kakeya sets achieved by Katz--Zahl. For that reason, we instead apply directly to their two-ends estimates and obtain a strictly better result. Finally, using the bootstrapped Katz--Zahl incidence estimate, we are able through a diagonal interpolation to get the claimed Kakeya maximal estimate improvement. 

Before beginning the argument, we indicate why a particular parameter choice will appear. Eliminating the intermediate scale $\rho$ from the interpolation of $G,P,H$ and further study of eligibility for a bootstrap iteration leads to the one-parameter family
\begin{equation}
\label{GPHinterpolation}
G^\sigma P^{6\sigma}H^{1-7\sigma},
\qquad \frac{1}{14}\le \sigma\le \frac{1}{7}.
\end{equation}

After one application of the X-ray bootstrap, the resulting dimensional parameter is strictly improved as $\sigma$ decreases, so the optimal one-step choice is $\sigma=\frac{1}{14}$. 
\newpage
This endpoint continues to be the optimizing one on the second-step bootstrap, where the dimensional parameter is improved by the decreasing of $\sigma$ as well. Thus, one just needs to show that the second iteration is better than the zeroth and the first, and that there is no possible combination of exponents to get a third.

\subsubsection{Twice X-ray bootstrap application} 
\label{twiceXray}
Interpolate the $G,P,H$-estimates with exponents $\frac{1}{14}$, $\frac{3}{7}$, and $\frac{1}{2}$, respectively. The $\rho$-powers cancel and the remaining estimate, after excluding the $q$-factor, is 
\begin{equation}
\mu_1
\lessapprox
C_1^{\frac{C}{\varepsilon_1}}
\lambda_1^{-\frac{55}{56}}
\delta^{-\frac{55}{56}}
X_1^{\frac{55}{168}}.
\end{equation}
After the standard refinement transfer \eqref{refinementdensity} and double counting \eqref{vera}, we get
$\mathbf{RT}
\left(
\frac{169}{56},
\frac{111}{56},
\frac{113}{168}
\right)$.
Since
$
\frac{113}{168}
=
\frac13
\left(
\frac{169}{56}-1
\right),
$
Proposition~\ref{RTtoTE} then implies that the assumption
$\mathbf{TE}
\left(
\frac{169}{56},
\frac{111}{56},
\frac{113}{168}
\right)$ is valid.
Choose
$a_{0}=\frac{55}{56}$,
$\ell_{0}=\frac{55}{56}$ and
$x_{0}=\frac{55}{168}$.
Note that
$
a_{0}-1+\frac{4x_{0}}{7}
=
\frac{199}{1176}>0,
$
and hence
$
w_{0}
=
\frac{1176}{1375}.
$
One X-ray iteration gives
$a_{1}=\frac{5217}{5500}$,
$\ell_{1}=\frac{9051}{5500}$, and
$x_{1}=\frac{2671}{16500}$.
Moreover,
$
x_{1}
=
\frac{2671}{16500}
<
\frac{5217}{16500}
=
\frac{a_{1}}{3}$.
Therefore,
$\mathbf{TE}
\left(
\frac{16783}{5500},
\frac{14551}{5500},
\frac{13829}{16500}
\right)$ holds.
We may iterate it once more. Indeed,
$
a_{1}-1+\frac{4x_{1}}{7}
=
\frac{431}{10500}>0.
$
Thus,
$w_{1}=\frac{10500}{10931}$. The second iteration yields
$a_{2}=\frac{452451}{480964}$,
$\ell_{2}=\frac{84651}{43724}$, and
$x_{2}=\frac{85405}{1442892}$.
Again,
$
x_{2}
=
\frac{85405}{1442892}
<
\frac{452451}{1442892}
=
\frac{a_{2}}{3}.
$
Consequently,
we have
$\mathbf{TE}
\left(
\frac{1471405}{480964},
\frac{128375}{43724},
\frac{1357487}{1442892}
\right)$.
A third iteration is not admissible because
$
a_{2}-1+\frac{4x_{2}}{7}
=
-\frac{257153}{10100244}<0.
$
In particular, the remark done in Definition $1.2$ of \cite{BCCLXZ} implies that every Kakeya set $K\subset \mathbb{R}^4$ satisfies
\begin{equation}
\dim_{\mathbb{H}}K
\geq
\frac{1471405}{480964}
\approx
3.05928.
\label{eq:hausdorff-final}
\end{equation}
Like previously mentioned, this does not improve the result achieved by Katz--Zahl, which gives 
$3+\frac{\sqrt{17665}-97}{600}
\approx3.05984$ as a lower-bound. We shall bootstrap the \textbf{TE}-estimate presented in Lemma $40$ (more precisely, Equation $(111)$) of \cite{KZ} in order to obtain a surpassing result.

\subsubsection{X-ray bootstrap applied to Katz--Zahl}
\label{sec:KZ-xray-bootstrap}

We now apply Lemma~\ref{lem:general-xray-bootstrap} to the two-ends estimate of Katz--Zahl. Let $\alpha=\frac{257-\sqrt{17665}}{128}$. The above mentioned incidence estimate can be described as $\mathbf{TE}\left(
4-\alpha,\frac{75}{28},1-\frac{\alpha}{3}
\right)$. Choose $a_{0,KZ}=\alpha$, 
$\ell_{0,KZ}=\frac{47}{28},$ and $x_{0,KZ}=\frac{\alpha}{3}$, just as noted in Remark~\ref{specialization}.
Note that $0<x_{0,KZ}<\frac7{16}$ and $a_{0,KZ}-1+\frac{4x_{0,KZ}}7
=\frac{25\alpha-21}{21}>0$. Hence, the first X-ray bootstrap is admissible, with $w_{0,KZ}=\frac{21}{25\alpha}$. Applying Lemma~\ref{lem:general-xray-bootstrap}, we obtain $a_{1,KZ}=\frac{3(53\alpha-21)}{100\alpha}$,  $\ell_{1,KZ}=\frac{3(91\alpha-16)}{100\alpha}$ and  $x_{1,KZ}=\frac{241\alpha-189}{300\alpha}$. Moreover,
$\frac{a_{1,KZ}}3-x_{1,KZ}
=\frac{63-41\alpha}{150\alpha}>0$. Therefore, by Proposition~\ref{RTtoTE}, $\mathbf{TE}\left(4-a_{1,KZ}, 1+\ell_{1,KZ}, 1-x_{1,KZ} \right)$ holds. We may iterate once more. Indeed, $0<x_{1,KZ}<\frac7{16}$ and
$a_{1,KZ}-1+\frac{4x_{1,KZ}}7
= \frac{2203\alpha-2079}{2100\alpha}>0$.
Thus, $w_{1,KZ} = \frac{2100\alpha}{4303\alpha-2079}.$ The second iteration yields $a_{2,KZ} = \frac{3(6655\alpha-3843)}
{4(4303\alpha-2079)}$, $\ell_{2,KZ}
=\frac{21(2587\alpha-1515)}
{4(4303\alpha-2079)}$, and $x_{2,KZ}
= \frac{23683\alpha-21735} {12(4303\alpha-2079)}$.
Again, $\frac{a_{2,KZ}}3-x_{2,KZ}
= \frac{5103-1859\alpha}
{6(4303\alpha-2079)}>0$. Consequently, $\mathbf{TE}\left(
4-a_{2,KZ}, 1+\ell_{2,KZ}, 1-x_{2,KZ}
\right)$ holds. Substituting the value of $\alpha$, we find that the assumption
$\mathbf{TE}( 3.06500\ldots,3.49127\ldots, 0.95123\ldots)$ holds. A third X-ray iteration is not admissible since $a_{2,KZ}-1+\frac{4x_{2,KZ}}7
= \frac{152545\alpha-154413}
{84(4303\alpha-2079)}
<0$. This gives the improvement we wanted by \cite{BCCLXZ}: for every Kakeya set $K\subset\mathbb R^4$,
\begin{equation}
\label{eq:KZ-hausdorff-final}
\dim_{\mathbb H}K
\ge
4-a_{2,KZ}
=
\frac{
46480643+10717\sqrt{17665}
}{
15629688
}
\approx
3.065.
\end{equation}

\subsubsection{Final diagonal interpolation}
\label{diagonalinterpolation}
The estimate $\mathbf{TE}\left(
\frac{1471405}{480964},
\frac{128375}{43724},
\frac{1357487}{1442892}
\right)$ obtained in Section~\ref{twiceXray} has a $\lambda$-exponent slightly smaller than its dimensional exponent. To recover an estimate on the diagonal, we interpolate it with the twice-bootstrapped Katz--Zahl estimate, whose $\lambda$-exponent is larger than its dimensional exponent.
Write $\alpha_1=\frac{452451}{480964}$, $\beta_1=\frac{128375}{43724}$, $b_1=\frac{1357487}{1442892}$, and $\alpha_2=
\frac{3(6655\alpha-3843)}
{4(4303\alpha-2079)}$, $\beta_2=
\frac{13(5503\alpha-3087)}
{4(4303\alpha-2079)}$, $b_2=
\frac{27953\alpha-3213}
{12(4303\alpha-2079)}$. Thus,
$\mathbf{TE}(4-\alpha_1,\beta_1,b_1)$ and
$\mathbf{TE}(4-\alpha_2,\beta_2,b_2)$ hold. Note that $\alpha_1+\beta_1 = \frac{466144}{120241}<4$, while
$\alpha_2+\beta_2 = 4+
\frac{3(1888\alpha-1533)}
{4303\alpha-2079}>4$. Choose $\theta=
\frac{\alpha_2+\beta_2-4}
{(\alpha_2+\beta_2)-(\alpha_1+\beta_1)}=
\frac{
120241(1888\alpha-1533)
}{
21(11822468\alpha-9266653)
}.$ Then, $\theta(\alpha_1+\beta_1)
+ (1-\theta)(\alpha_2+\beta_2)=4$.
Interpolating the two-ends incidence estimates gives
\[
\alpha_\theta = \theta\alpha_1+(1-\theta)\alpha_2 = \frac{ 3(15124676\alpha-11913661)
}{4(11822468\alpha-9266653)},
\]
\[
\beta_\theta = \theta\beta_1+(1-\theta)\beta_2 = \frac{ 65(2212084\alpha-1731161)}{ 4(11822468\alpha-9266653)},
\]
and
\[
b_\theta = \theta b_1+(1-\theta)b_2 = \frac{128620156\alpha-99852371}{
12(11822468\alpha-9266653)}.
\]
By construction,
$\alpha_\theta+\beta_\theta=4$. Therefore, we find that the two-ends assumption $\mathbf{TE}\left(
3.06056\ldots,
3.06056\ldots,
0.94314\ldots
\right)$ is valid, by explicit computations over $\alpha$. Finally, the remark in Definition~$1.2$ of \cite{BCCLXZ} (or Proposition~$11$ of \cite{KZ}) implies that the Kakeya maximal estimate holds at dimension
\[
\frac{
65(2212084\alpha-1731161)
}{
4(11822468\alpha-9266653)
}
\approx3.06056,
\qquad
\alpha=\frac{257-\sqrt{17665}}{128}.
\]
This finishes the proof of Theorem~\ref{vera3}.

\subsection{Sharpness of the interpolation-bootstrap technique}
\label{bootstrapsharpness}

We now show that the two X-ray branches used above were not chosen randomly. Within the family generated by scale-free interpolation of the estimates $G,P,H$, followed by repeated applications of Lemma~\ref{lem:general-xray-bootstrap} whenever admissible, they are the optimizing endpoints.

\subsubsection{Sharpness of the X-ray branches}

Consider the following interpolation that allows the cancellation between the exponents of $\rho$:
$$
G^\sigma P^{6\sigma}H^{1-7\sigma},
\qquad
0\leq\sigma\leq\frac{1}{7}.
$$
After carefully calculating the exponents for each term and eliminating the factor $q^{-\frac{7\sigma}{4}}\lessapprox 1$, we end up with
\begin{equation}
\mu_1
\lessapprox
C_1^{\frac{C}{\varepsilon_1}}
\lambda_1^{-\ell_0(\sigma)}
\delta^{-a_0(\sigma)}
X_1^{x_0(\sigma)},
\label{eq:optimized-seed-family}
\end{equation}
where
\begin{equation}
a_0(\sigma)
=
1-\frac{\sigma}{4},
\qquad
\ell_0(\sigma)
=
\frac{3}{4}+\frac{13\sigma}{4},
\qquad
x_0(\sigma)
=
\frac{1}{2}-\frac{29\sigma}{12}.
\label{eq:optimized-seed-parameters}
\end{equation}
By performing the now usual refinement transfer and double counting, we obtain that
$\mathbf{RT}
\left(
4-a_0(\sigma),
1+\ell_0(\sigma),
1-x_0(\sigma)
\right)
$ holds. The robust-transversality reduction is admissible precisely when
$x_0(\sigma)\leq\frac{a_0(\sigma)}{3}.$
Using \eqref{eq:optimized-seed-parameters}, this condition becomes
$\sigma\geq\frac{1}{14}.$ For $\sigma<\frac{1}{14}$, one can weaken $1-x_0(\sigma)$ to $1-\frac{a_0(\sigma)}{3}$ and still apply Proposition~\ref{RTtoTE}, but these estimates are dominated by $\sigma=\frac{1}{14}$. Consequently, the admissible range of scale-free interpolations available from $G,P,H$ is exactly the one in \eqref{GPHinterpolation}.
\newpage
See that, in this interval range, the first X-ray iteration is admissible since $a_0(\sigma)-1+\frac{4x_0(\sigma)}7
=
\frac{24-137\sigma}{84}>0.$
We obtain
\begin{equation}
a_1(\sigma)
=
\frac{3(136-165\sigma)}
{4(108-137\sigma)},
\qquad
\ell_1(\sigma)
=
\frac{21(36-73\sigma)}
{4(108-137\sigma)},
\qquad
x_1(\sigma)
=
\frac{312-1697\sigma}
{12(108-137\sigma)}.
\end{equation}
The output is a genuine two-ends estimate throughout
the full admissible interval. Indeed, $\frac{a_1(\sigma)}3-x_1(\sigma)
=
\frac{48+601\sigma}
{6(108-137\sigma)}>0.$
Therefore, we see that the assumption 
$
\mathbf{TE}
\left(
4-a_1(\sigma),
1+\ell_1(\sigma),
1-x_1(\sigma)
\right)$ holds
for every
$
\frac1{14}\leq\sigma\leq\frac17.
$
Moreover,
$$
\frac{d}{d\sigma}x_1(\sigma)
=
-\frac{11711}{(108-137\sigma)^2}<0
\quad\implies\quad
0<x_1(\sigma)\leq
x_1\left(\frac1{14}\right)
=
\frac{2671}{16500}<\frac7{16}.
$$
So, the $x$-range required by the bootstrap remains valid.

We now optimize the one-step dimensional exponent. Differentiating the $a_1(\sigma)$-expression, $$
\frac{d}{d\sigma}a_1(\sigma)
=
\frac{609}{(108-137\sigma)^2}>0.$$ Hence, $4-a_1(\sigma)$ is strictly decreasing in the admissible interval. The strongest one-step dimensional estimate occurs at $\sigma=\frac{1}{14}$, which is the parameter choice used in Section~\ref{twiceXray}.

To apply the X-ray bootstrap for the second time, we require
$
a_1(\sigma)-1+\frac{4x_1(\sigma)}7\geq0$,
which gives the interval $\frac{1}{14}
\leq
\sigma
\leq
\frac{744}{5675}$ 
where a second iteration is available. A second application of
Lemma~\ref{lem:general-xray-bootstrap} gives
\begin{equation}
\label{secondapplication}
\resizebox{0.99\textwidth}{!}{$\displaystyle
    a_2(\sigma)
=
\frac{
3(12168-19535\sigma)
}{
4(9816-17183\sigma)
},\,\,
\ell_2(\sigma)
=
\frac{
21(4056-12443\sigma)
}{
4(9816-17183\sigma)
},\,\,
x_2(\sigma)
=
\frac{
11688-78227\sigma
}{
12(9816-17183\sigma)
}.
$}
\end{equation}
Also, see that $\frac{a_2(\sigma)}3-x_2(\sigma) = \frac{12408+9811\sigma}
{6(9816-17183\sigma)}>0$,
and therefore the assumption
$
\mathbf{TE}
\left(
4-a_2(\sigma),
1+\ell_2(\sigma),
1-x_2(\sigma)
\right)
$
holds throughout the twice-iterable interval.

We now optimize the twice-iterated dimensional exponent. Differentiating the $a_2(\sigma)$-expression, $$\frac{d}{d\sigma}a_2(\sigma)
=
\frac{
12995388
}{
(9816-17183\sigma)^2
}
>0.$$
Hence, $4-a_2(\sigma)$ is strictly decreasing. Thus, the strongest twice-iterated dimensional estimate occurs at $\sigma=\frac{1}{14}$, which is once again the parameter choice used in Section~\ref{twiceXray}. This somehow explains why there was no distinction over iterating once or twice. Moreover, we can prove that $\frac{1471405}{480964}$ is the best dimensional exponent among the zero-, one-, and two-step families. Indeed, the best uniterated value is
$$
\min_{0\leq\sigma\leq \frac{1}{7}}a_0(\sigma)=a_0\left(\frac{1}{7}\right)=\frac{27}{28},
$$
the best once-iterated value is
$$
\min_{\frac{1}{14}\leq \sigma\leq \frac{1}{7}} a_1(\sigma)=a_1\left(\frac{1}{14}\right)=\frac{5217}{5500},
$$
and the best twice-iterated value is
$$
\min_{\frac{1}{14}
\leq
\sigma
\leq
\frac{744}{5675}} a_2(\sigma)=a_2\left(\frac{1}{14}\right)=\frac{452451}{480964}.
$$
Hence, among the endpoints under consideration, the twice-iterated endpoint yields the largest dimensional exponent. This motivates its use in the final diagonal interpolation in Section~\ref{diagonalinterpolation}, where the resulting bound depends both on the dimensional exponent $4-a$ and on the signed difference $3-a-\ell$ between the dimensional exponent and the $\lambda$-exponent. It remains to show that no admissible branch can be continued for the third time. Using \eqref{secondapplication},
\begin{equation}
a_2(\sigma)-1+\frac{4x_2(\sigma)}7
=
-\frac{
11208+100241\sigma
}{
84(9816-17183\sigma)
}
<0,
\,\,
\text{for every }
\frac{1}{14}
\leq
\sigma
\leq
\frac{744}{5675}.
\label{eq:r2-gamma}
\end{equation}
Thus, no twice-iterated estimate in this interpolation family can be used as the input for a third application of Lemma~\ref{lem:general-xray-bootstrap}. This ends the optimization argument for Section  \ref{twiceXray}.
\begin{remark}
    One can show that adding the estimates $L$ and $H$ in the optimization method presented in Section~\ref{diagsharpness} does not upgrade the restriction range. Similarly, the introduction of the $R$-estimate in the interpolation technique of Sections \ref{offsharpness} and \ref{bootstrapsharpness} does not give a better off-diagonal estimate or an improved dimension in the Kakeya maximal estimate. Thus, one needs to derive smarter combinations of meaningful interpolation estimates (for instance, introducing interpolation in between bootstrap steps) or new estimates for $\mu_1$ in our imposed assumptions not present in the current literature (at least known by the author) in order to obtain better results for both problems in question.
\end{remark}

\printbibliography

\end{document}